\documentclass[english]{article}
\usepackage[T1]{fontenc}
\usepackage[latin9]{inputenc}
\usepackage{color}
\usepackage{babel}
\usepackage{verbatim}
\usepackage{mathrsfs}
\usepackage{url}
\usepackage{tipa}
\usepackage{amsmath}
\usepackage{amsthm}
\usepackage{amssymb}
\usepackage{graphicx}
\usepackage[unicode=true,pdfusetitle,
 bookmarks=true,bookmarksnumbered=false,bookmarksopen=false,
 breaklinks=false,pdfborder={0 0 1},backref=false,colorlinks=true]
 {hyperref}

\makeatletter

\providecommand{\tabularnewline}{\\}

\newcommand{\lyxaddress}[1]{
	\par {\raggedright #1
	\vspace{1.4em}
	\noindent\par}
}
\theoremstyle{remark}
\newtheorem{rem}{\protect\remarkname}
\theoremstyle{definition}
\newtheorem{condition}{\protect\conditionname}
\theoremstyle{plain}
\newtheorem{assumption}{\protect\assumptionname}
\theoremstyle{definition}
 \newtheorem{example}{\protect\examplename}
\theoremstyle{plain}
\newtheorem{thm}{\protect\theoremname}
\theoremstyle{plain}
\newtheorem{cor}{\protect\corollaryname}
\theoremstyle{plain}
\newtheorem{lem}{\protect\lemmaname}
\theoremstyle{definition}
\newtheorem{defn}{\protect\definitionname}

\makeatother

\providecommand{\assumptionname}{Assumption}
\providecommand{\conditionname}{Condition}
\providecommand{\corollaryname}{Corollary}
\providecommand{\definitionname}{Definition}
\providecommand{\examplename}{Example}
\providecommand{\lemmaname}{Lemma}
\providecommand{\remarkname}{Remark}
\providecommand{\theoremname}{Theorem}

\begin{document}
\title{Subgeometric hypocoercivity for piecewise-deterministic Markov process
Monte Carlo methods}
\author{Christophe Andrieu\thanks{\protect\url{c.andrieu@bristol.ac.uk}}\and
Paul Dobson\thanks{\protect\url{p.dobson@tudelft.nl}}\and Andi Q.
Wang\thanks{\protect\url{andi.wang@bristol.ac.uk}}}
\maketitle

\lyxaddress{School of Mathematics, University of Bristol, United Kingdom.}

\lyxaddress{Delft Institute of Applied Mathematics, Delft University of Technology,
The Netherlands.}
\begin{abstract}
We extend the hypocoercivity framework for piecewise-deterministic
Markov process (PDMP) Monte Carlo established in \cite{Andrieu18}
to heavy-tailed target distributions, which exhibit subgeometric rates
of convergence to equilibrium. We make use of weak Poincaré inequalities,
as developed in the work of \cite{Grothaus19}, the ideas of which
we adapt to the PDMPs of interest. On the way we report largely potential-independent
approaches to bounding explicitly solutions of the Poisson equation
of the Langevin diffusion and its first and second derivatives, required
here to control various terms arising in the application of the hypocoercivity
result.
\end{abstract}

\section{Introduction}

\global\long\def\R{\mathbb{R}}%
\global\long\def\dif{\mathrm{d}}%

\global\long\def\X{\mathsf{X}}%
\global\long\def\E{\mathsf{E}}%
\global\long\def\V{\mathsf{V}}%
\global\long\def\F{\mathsf{\mathscr{F}}}%

\global\long\def\L{\mathcal{L}}%
\global\long\def\S{\mathcal{S}}%
\global\long\def\T{\mathcal{T}}%
\global\long\def\D{\mathrm{D}}%
\global\long\def\C{\mathsf{C}}%
\global\long\def\A{\mathcal{A}}%
\global\long\def\Ab{\mathcal{\bar{\A}}}%
\global\long\def\ctb{\mathsf{\mathrm{C}}_{{\rm b}}^{2}}%

\global\long\def\ELL{\mathrm{L}^{2}}%
\global\long\def\osc{\mathrm{osc}}%
\global\long\def\Id{\mathrm{Id}}%
\global\long\def\refr{\mathrm{ref}}%
\global\long\def\Tr{\mathrm{Tr}}%
\global\long\def\|{\mathcal{\lVert}}%

\global\long\def\Rv{\mathcal{R}_{v}}%
\global\long\def\bounce{\mathcal{B}}%

In this work, we study piecewise-deterministic Markov processes (PDMPs)
which are used in the context of Monte Carlo inference to draw samples
from some given target density $\pi$ on $\R^{d}$, for instance in
Bayesian computation. Notable examples of such processes are the Zig-zag
process of \cite{Bierkens19} and the Bouncy Particle Sampler of \cite{BouchardCote18}.
PDMPs have gained attention within the field of Markov chain Monte
Carlo (MCMC) because such methods depart significantly from traditional
reversible MCMC approaches based on Metropolis--Hastings. These processes
are constructed to be nonreversible, intuitively enabling persistent
exploration within the state space, rather than the diffusive exploration
characteristic of reversible schemes.

However, this nonreversiblility also introduces mathematical difficulties
when analysing the theoretical properties of the resulting algorithms,
such as rates of convergence. Traditional methods based on spectral
theory for self-adjoint operators in Hilbert spaces can no longer
be applied, and furthermore the underlying operators which define
the process tend to be non-coercive: the symmetric component of the
operator has a nontrivial kernel. This implies that one cannot expect
straightforward geometric convergence of the semigroup $(P_{t})$,
in the sense that there exists some $\rho>0$, such that for appropriate
functions $f$, 
\begin{equation}
\| P_{t}f\|\le C\| f\|\xi(t),\quad\forall t\ge0,\label{eq:exp_conv}
\end{equation}
for some $C\le1$, $\xi(t)=\exp(-\rho t)$ for $t\ge0$ and an appropriate
norm $\|\cdot\|$. In order to understand degenerate dynamics, the
hypocoercivity framework has been developed, following the approach
of \cite{Dolbeault15}.

This framework was first applied to PDMPs in \cite{Andrieu18}, where
exponential convergence of the semigroups was proven as in (\ref{eq:exp_conv})
when the target density $\pi$ satisfies a Poincaré inequality: for
some constant $C_{\mathrm{P}}>0$,
\begin{equation}
\|\nabla f\|_{2}^{2}\ge C_{\mathrm{P}}\| f\|_{2}^{2},\label{eq:poincare_intro}
\end{equation}
for suitably differentiable functions $f\in\ELL(\pi)$, with $\int f\,\dif\pi=0$,
where $\|\cdot\|_{2}$ is the norm in $\ELL(\pi)$. The authors were
able to conclude that (\ref{eq:exp_conv}) holds for such targets,
with an exponential rate function $\xi$, and for some constant $C>1$.

The goal of this work is to extend the hypocoercivity results of \cite{Andrieu18}
to targets $\pi$ which do not possess a Poincaré inequality (\ref{eq:poincare_intro}),
but instead possess a weak Poincaré inequality of the form
\[
\| f\|_{2}^{2}\le\alpha(r)\|\nabla f\|_{2}^{2}+r\Psi(f),\quad\forall r>0,
\]
where $\alpha:(0,\infty)\to[1,\infty)$ is a decreasing function,
typically divergent as $r\downarrow0$, and $\Psi$ is an appropriate
functional. This encompasses target distributions which possess subgeometric
tail decay, and are typically referred to as `heavy-tailed'. To do
this, we will utilize the approach of \cite{Grothaus19}, where such
inequalities were studied for degenerate diffusions. Our main abstract
result will be a convergence result of the form (\ref{eq:exp_conv})
where the rate function $\xi$ is in fact subgeometric. 

As a concrete application, our bounds on the semigroup of the form
(\ref{eq:exp_conv}) will allow us to check conditions which ensure
that a central limit theorem holds for (appropriately scaled) ergodic
averages of the process.

\subsection{Contribution}

In this subsection, we carefully describe our contributions in relation
to the literature, particularly references \cite{Andrieu18,Grothaus19}
and \cite{Grothaus14sti}. Readers interested in our actual results
are encouraged to move on to the following subsection where we define
our notation, or to Section \ref{sec:PDMP-assumptions} for our assumptions,
Section \ref{sec:Abstract-result} for our abstract result, Section
\ref{sec:Application-to-PDMPs} for our result for PDMPs, or to Section
\ref{sec:Examples} where we compare our bounds theoretically and
empirically on illustrative examples.

In the present manuscript, we work in and extend the general framework
laid out in \cite{Andrieu18} for PDMPs. The analysis carried out
in \cite{Andrieu18} crucially relied on the existence of a (strong)
Poincaré inequality (\ref{eq:poincare_intro}), which enabled the
authors to establish geometric convergence of the semigroup (\ref{eq:exp_conv}).
These results rely themselves on the framework proposed by \cite{Dolbeault15}
for which the first rigorous proof was established in \cite{Grothaus14sti}
and whose results were adapted to take into account technical specificities
of PDMPs in \cite{Andrieu18}. Our work aims to combine the framework
recently proposed in \cite{Grothaus19} to tackle scenarios where
application of the ideas of \cite{Dolbeault15,Grothaus14sti} is sought,
but only a weak form of the Poincaré inequality is satisfied, and
\cite{Andrieu18} which takes into account PDMP idiosyncrasies.

More specifically, while our abstract assumptions and the resulting
theorem and its proof (see Sections \ref{sec:PDMP-assumptions}, \ref{sec:Abstract-result})
may superficially appear very similar to those of \cite{Grothaus19},
we were not able to straightforwardly apply their results and have
adapted them following \cite{Andrieu18}. This disparity fundamentally
arises from the differences in the how the corresponding processes
arise. For PDMPs, as in the present work and in \cite{Andrieu18},
the initial point of departure is an explicit construction of the
process, in terms of the deterministic dynamics and the switching
mechanism, driven by an inhomogeneous Poisson process. The infinitesimal
generator is a by-product, and its closure not sufficiently tractable
to work with. On the other hand, for diffusion processes as in \cite{Grothaus19},
one can begin with an appropriate differential operator, the putative
infinitesimal generator, take the closure, and then via the standard
operator-theoretic machinery define the ensuing semigroup and stochastic
process. The key technical differences arise on the level of checking
the closure of certain operators and we comment on this point in relevant
places in the text. 

In relation to the actual results of \cite{Grothaus19}, by specialising
to our particular PDMP setting we are also able to obtain slightly
better constants in the decay of the semigroup. In relation to \cite{Andrieu18},
by leveraging the powerful results of \cite{Lorenzi06}, we establish
(quasi) potential-independent approaches to bounding the difficult
cross terms arising in the application of the hypocoercivity result,
which depend on smoothness estimates of the solution of the Poisson
equation for the Langevin diffusion process.

The recent work \cite{Lu2020} also studies geometric convergence
of PDMP semigroups. Their work also crucially relies on a (strong)
Poincaré inequality, and their Assumption 3 typically holds when the
potential $U(x)=-\log\pi(x)$ grows at a superlinear rate in $|x|$.
Thus the framework of \cite{Lu2020} cannot currently be applied to
heavy-tailed targets.

Finally, we briefly mention some other related work. For an accessible
introduction to hypocoercivity, we recommend \cite{Achleitner2019},
which focusses on the finite-dimensional ODE setting. In \cite{Deligiannidis2018},
hypocoercivity techniques were used to study randomized Hamiltonian
Monte Carlo (RHMC) and derive dimension-free exponential convergence
rates. For connections between hypocoercivity and convergence proofs
based on Lyapunov functions, see \cite{Monmarche2019}. For a broad
and recent review of the current literature on hypocoercivity, see
\cite{Bernard2020}.

\subsection{Notation}
\begin{itemize}
\item $\lvert\cdot\rvert$ denotes the Euclidean norm on $\R^{d}$ and for
$v,w\in\R^{d}$, $\langle v,w\rangle=v^{\top}w$ is the associated
inner product where $v^{\top}$ is the transpose of $v$. 
\item $\mathrm{I}_{d}$ denotes the $d\times d$ identity matrix. 
\item For a vector $w\in\R^{d}$ we will write $w_{i}$, $i=1,\dots,d$,
for its coordinates with respect to the standard basis.
\item For $R>0$, $B_{R}:=\{x\in\R^{d}\colon|x|\leq R\}$. 
\item For any $s\in\R$, $(s)_{+}:=\max\{s,0\}$ denotes the positive part.
\item For $A$ a set, $\mathbb{I}A$ is the associated indicator or characteristic
function.
\item For a smooth manifold $\mathsf{M}$ and $k\in\mathbb{N}\cup\{\infty\}$,
$\mathrm{C}^{k}(\mathsf{M},\R^{m})$ denotes the set of $k$-times
continuously differentiable functions $f:\mathsf{M}\to\mathbb{R}^{m}$.
$\mathsf{\mathrm{C}}^{k}(\mathsf{M})$ denotes $\mathrm{C}^{k}(\mathsf{M},\R)$.
$\mathsf{\mathrm{C}}_{b}^{k}(\mathsf{M})$ denotes functions in $\mathsf{\mathrm{C}}^{k}(\mathsf{M})$
which are in addition bounded and have bounded derivatives up to order
$k$. A subscript $\mathsf{\mathrm{C}}_{\mathrm{c}}^{k}(\mathsf{M})$
denotes functions in $\mathsf{\mathrm{C}}^{k}(\mathsf{M})$ which
are compactly supported.
\item For $\alpha\in(0,1)$, $\mathrm{C}^{k+\alpha}(\mathsf{M},\R^{m})$
denotes the set of $k$-times continuously differentiable functions
with locally $\alpha$-Hölder continuous $k$-th derivative.
\item For $f\in\mathsf{\mathrm{C}}^{k}(\mathsf{M})$, $i\in\{1,2,\dots,d\}$,
$x\mapsto\partial_{i}f(x)$ denotes the partial derivative of $f$
with respect to the $i$th coordinate, for $k\ge1$, and analogously
for $i,j\in\{1,2,\dots,d\}$, $\partial_{i,j}f$ denotes $\partial_{i}\partial_{j}f$,
for $k\ge2$.
\item For a function $f\in\mathrm{C}^{2}(\mathsf{X})$, $\nabla^{2}f$ denotes
the Hessian matrix of second-order partial derivatives. For $f\in\mathrm{C}^{2}(\mathsf{X})$,
$\Delta_{x}f:=\sum_{i=1}^{d}\partial_{i,i}f$ denotes the Laplacian.
\item For $f=(f_{1},\dots f_{m})\in\mathrm{C}^{k}(\mathsf{M},\R^{m})$,
$\nabla_{x}f$ is the gradient of $f$, defined at any $x\in\mathsf{M}$
by $\nabla_{x}f(x)=(\partial_{j}f_{i}(x))_{i\in\{1,\dots,m\},j\in\{1,\dots,d\}}\in\R^{d\times m}$. 
\item For any measurable space $(\mathsf{M},\mathcal{F})$ with probability
measure $\mathrm{m}$, we let $\ELL(\mathrm{m})$ be the Hilbert space
of real measurable functions $f$ with $\int_{\mathsf{M}}|f|^{2}\dif\mathrm{m}<\infty$,
with inner product $\langle f,g\rangle_{2}=\int_{\mathsf{M}}fg\:\dif\mathrm{m}$
and corresponding norm $\|\cdot\|_{2}$. When there is ambiguity we
may also write $\langle f,g\rangle_{\ELL(\mathrm{m})}$, $\|\cdot\|_{\ELL(\mathrm{m})}$
or $\langle f,g\rangle_{\mathrm{m}}$, $\|\cdot\|_{\mathrm{m}}$.
We use the same notation for $\langle F,G\rangle_{\mathrm{m}}=\int_{\mathsf{M}}{\rm Tr}\big(F^{\top}G\big)\dif\mathrm{m}$
with $F,G\colon\mathsf{M}\rightarrow\mathbb{R}^{d}$.
\item $\Id:\ELL(\mathrm{m})\to\ELL(\mathrm{m})$ denotes the identity mapping,
$f\mapsto f$.
\item We let $\mathrm{H}^{1}(\mathsf{M},\R^{m},{\rm m})=\{g\in\mathsf{\mathrm{C}}^{1}(\mathsf{M},\R^{m})\cap\ELL({\rm m})\colon\|\nabla_{x}g\|_{2}<\infty\}$,
$\mathrm{H}^{1}(\mathsf{M}):=\mathrm{H}^{1}(\mathsf{M},\R,{\rm m})$
and $\mathrm{H}^{2}(\mathsf{M}):=\{g\in\mathrm{H}^{1}(\mathsf{M})\cap\mathsf{\mathrm{C}}^{2}(\mathsf{M})\colon\|\nabla_{x}^{2}g\|_{2}<\infty\}$.
Note that we use the notation normally associated with Sobolev spaces,
but our derivatives are not weak derivatives.
\item For a measurable function $f:\mathsf{M}\to\R$, let $\| f\|_{\osc}:=\mathrm{ess_{\mathrm{m}}}\sup f-\mathrm{ess}_{\mathrm{m}}\inf f$.
\item Let $\mathrm{L}_{\mathrm{m}}^{2}(\mathsf{M};\R^{d})$ be the space
of functions $f:\mathsf{M}\to\R^{d}$ with $\lVert h\rVert_{\mathrm{L}_{\mathrm{m}}^{2}(\mathsf{M};\R^{d})}^{2}<\infty$,
endowed with the norm
\end{itemize}
\[
\lVert h\rVert_{\mathrm{L}_{\mathrm{m}}^{2}(\mathsf{M};\R^{d})}^{2}=\int_{\mathsf{M}}\lvert h(x)\rvert^{2}\dif\mathrm{m}(x).
\]

\begin{itemize}
\item $\mathrm{L}^{\infty}(\mathrm{m})$ will denote the Banach space of
(equivalence classes of) measurable functions $f:\mathsf{M}\to\R$
with $\mathrm{ess_{\mathrm{m}}}\sup|f|<\infty$.
\item For an unbounded operator $(\A,\D(\A))$, we let $\mathrm{Ran}(\A):=\{\A f:f\in\D(\A)\}$
and $\mathrm{Ker}(\A):=\{f\in\D(\A):\A f=0\}$. If $\A$ is closable,
we let $\bar{\A}$ denote its closure.
\end{itemize}

\subsection{PDMP notation}

We summarize here our PDMP notation; for the underlying assumptions,
see Section \ref{sec:PDMP-assumptions}. Given potential $U:\X\to\R$,
we will denote the target distribution of interest by $\pi=\mathrm{e}^{-U}/\int_{\X}\mathrm{e}^{-U(y)}\dif y$
on $\X=\R^{d}$ equipped with its Borel $\sigma$-algebra. $\V\subset\R^{d}$
is a closed subset and we have a probability measure $\nu$ defined
on $\V$ equipped with its Borel $\sigma$-algebra $\mathcal{V}$.
Then set $\E=\X\times\V$ and define the augmented probability measure
$\mu=\pi\otimes\nu$. We will be working with PDMPs whose generators
are of the form, for $f\in\mathrm{C}_{b}^{2}(\E)$, $(x,v)\in\E$,
\begin{align}
\L f(x,v) & =v^{\top}\nabla_{x}f(x,v)+\sum_{k=1}^{K}\lambda_{k}(x,v)[(\bounce_{k}-\Id)f](x,v)+m_{2}^{1/2}\lambda_{\refr}\Rv f(x,v).\label{eq:def-calL}
\end{align}
Here $\Rv$ is the refreshment operator, given for any $f\in\ELL(\mu)$,
by
\[
\Rv f:=\Pi_{v}f-f\text{ with }\Pi_{v}f(x,v):=\int_{\V}f(x,w)\:\dif\mathsf{\nu}(w).
\]
For a sequence of continuous vector fields $F_{k}:\X\to\R^{d}$, $k\in\{1,2,\dots,K\}$,
such that $\nabla_{x}U=\sum_{k=1}^{K}F_{k}$, we now define the corresponding
bounce operators $\bounce_{k}$. For each $k\in\{1,2,\dots,K\}$,
$x\in\X$ set 
\begin{equation}
\mathrm{n}_{k}(x)=\begin{cases}
F_{k}(x)/|F_{k}(x)| & |F_{k}(x)|\neq0,\\
0 & \mathrm{else},
\end{cases}\label{eq:nkx}
\end{equation}
then set for any $f:\E\to\R$, $(x,v)\in\E$,
\[
\bounce_{k}f(x,v)=f\left(x,v-2(v^{\top}\mathrm{n}_{k}(x))\mathrm{n}_{k}(x)\right).
\]
The intensity $\lambda_{k}(x,v)$ has an explicit form, depending
on the dynamics, which ensures invariance of $\mu$. In particular,
we require that $\lambda_{k}(x,v)-\lambda_{k}(x,-v)=v^{\top}F_{k}(x)$,
which is a necessary condition for $\mu$ to be an invariant measure
(see our Assumption \ref{assu:event_rate}). Finally
\[
m_{2}:=\int_{\V}v_{1}^{2}\:\dif\nu(v),
\]
assumed to be finite. 
\begin{rem}
The PDMPs considered in \cite{Andrieu18} are slightly more general
as they include the possibility of non-linear drift, that is the first
order derivative term

\[
v^{\top}\nabla_{x}f(x,v)-F_{0}(x)^{\top}\nabla_{v}f(x,v),
\]
This includes important examples such as randomized HMC \cite{BouRabee}
and the Boomerang Sampler \cite{Bierkens20}; to simplify the expressions
we have not included this term but the results hold in this more general
setting also. More complex refreshment operators can also be considered;
see \cite{Andrieu18}.
\end{rem}

\section{PDMP assumptions}

\label{sec:PDMP-assumptions}

In what follows, `Conditions' will refer to the conditions needed
for the abstract hypocoercivity result to hold; these are inspired
by \cite{Grothaus19}. `Assumptions' will refer to the assumptions
made on our PDMP process, which we will show imply that the Conditions
hold. We first give a basic condition, following \cite{Andrieu18}.
\begin{condition}
\label{cond:basic-core}We have:
\begin{enumerate}
\item The operator $(\mathcal{L},\D(\L))$ is closed in $\ELL(\mu)$ and
generates a strongly continuous contraction semigroup $(P_{t})_{t\ge0}$
on $\ELL(\mu)$.
\item $\mu$ is a stationary measure for $(P_{t})_{t\ge0}$.
\item There exists a core $\mathsf{C}$ for $\mathcal{L}$ such that $\mathsf{C}$
is dense in $\ELL(\mu)$ and $\mathsf{C}\subset\mathrm{D}(\mathcal{L})\cap\mathrm{D}(\mathcal{L}^{*})$,
where $(\mathcal{L}^{*},\mathrm{D}(\mathcal{L}^{*}))$ is the adjoint
of $\mathcal{L}$ on $\ELL(\mu)$.
\end{enumerate}
\end{condition}
We now give the assumptions on the potential $U$.
\begin{assumption}
\label{assu:U}The potential $U$ is such that,
\begin{enumerate}
\item \label{enu:UisC2plusalpha}$U\in\mathrm{C}^{2+\alpha}(\mathsf{X})$
for some $\alpha\in(0,1)$;
\item \label{enu:HessianUlowerbounded}there exists a constant $c_{U}\ge0$
such that for each $x\in\X$, $\nabla_{x}^{2}U(x)\succeq-c_{U}\mathrm{I}_{d}$
in the sense of definiteness of matrices;
\item \label{enu:nablaU-summable}either of the following holds:
\begin{enumerate}
\item $\nabla_{x}U$ is bounded, 
\item \label{enu:cond-laplacian-U} $\|\nabla_{x}U\|_{\pi}<\infty$ and
for some $C_{U}>0$ and $\omega\ge0$,
\begin{equation}
\Delta_{x}U(x)\leq C_{U}d^{1+\omega}+\lvert\nabla_{x}U(x)\rvert^{2}/2\,\text{ for all }\,x\in\X.\label{eq:boundonU}
\end{equation}
\end{enumerate}
\end{enumerate}
\end{assumption}
\noindent We remark that this is where our present work diverges
substantially from \cite{Andrieu18}. Since we are interested in studying
subgeometric rates of convergence, we do not assume a Poincaré inequality
here as in \cite{Andrieu18}: instead, we will later assume a weak
Poincaré inequality.
\begin{example}
\label{exa:targets}We shall consider two different heavy-tailed distributions:
\begin{enumerate}
\item Set $U(x)=\frac{1}{2}(d+p)\log\left(1+|x|^{2}\right)$, for some $p>0$,
in this case
\[
\pi(\dif x)=\frac{Z_{d,p}}{(1+|x|)^{d+p}}\dif x
\]
for some normalising constant $Z_{d,p}$;
\item Set $U(x)=\sigma|x|{}^{\delta}$, some $\sigma>0$ and $0<\delta<1$,
then 
\[
\pi(\dif x)=Z_{\sigma,\delta}e^{-\sigma|x|{}^{\delta}}\dif x
\]
 for some nomalising constant $Z_{\sigma,\delta}$.
\end{enumerate}
\end{example}
In both of these cases Assumption \ref{assu:U} holds since $\nabla_{x}U$
and $\nabla_{x}^{2}U$ are both bounded. Both of these examples have
subexponential decay but satisfy a weak Poincaré inequality as we
show in Example \ref{exa:targetscont} using the results of \cite{Rockner01}.
We assume the following assumptions on the vector fields, as in \cite{Andrieu18}.
\begin{assumption}
\label{assu:F_k}The family of vector fields $\{F_{k}:\X\to\R^{d};k\in\{0,1,,\dots,K\}\}$
satisfies:
\begin{enumerate}
\item for $k\in\{0,1,\dots,K\}$, $F_{k}\in\mathsf{\mathrm{C}}^{2}(\mathsf{X},\R^{d})$;
\item for all $x\in\X$, $\nabla_{x}U(x)=\sum_{k=1}^{K}F_{k}(x)$;
\item \label{enu:2(c)for-all-,}for all $k\in\{0,1,\dots,K\}$, there exists
$a_{k}\ge0$ such that for all $x\in\X$,
\[
|F_{k}(x)|\le a_{k}\{1+|\nabla_{x}U(x)|\}.
\]
\end{enumerate}
\end{assumption}
Following \cite{Andrieu18}, we make the following assumption on the
event rate.
\begin{assumption}
The events rates are given by $\lambda_{k}(x,v)=\varphi\left(v^{\top}F_{k}(x)\right)$,
for each $(x,v)\in\E$ and $k=1,2,\dots,K$, where $\varphi:\R\to\R_{+}$
is a continuous function satisfying for any $s\in\R$,
\[
\varphi(s)-\varphi(-s)=s,\quad|s|\le\varphi(s)+\varphi(-s)\le c_{\varphi}m_{2}^{1/2}+C_{\varphi}|s|,
\]
for some constants $c_{\varphi}\ge0,C_{\varphi}\ge1$.\label{assu:event_rate}
\end{assumption}
This assumption allows the canonical choices of rates, $\varphi(s)=(s)_{+}$
as well as smooth versions as in \cite{Andrieu19}, such as $\varphi=-\log\left(\phi(\exp(-s))\right)$,
for $\phi(r)=r/(1+r)$ or arbitrarily precise uniform approximation
of canonical rates.
\begin{example}
Many standard PDMP algorithms satisfy these assumptions:
\begin{enumerate}
\item Let $K=d$ and for $k\in\left\{ 1,\ldots,d\right\} ,x\in\X,F_{k}(x)=\partial_{k}U(x)e_{k}$,
where $e_{k}$ is the canonical basis, then we have the Zig-Zag process
\cite{Bierkens19}.
\item The choice $K=1$ and $F_{1}=\nabla_{x}U$ gives the Bouncy Particle
Sampler \cite{BouchardCote18}.
\end{enumerate}
\end{example}
Now we give assumptions on $\V$ and $\nu$:
\begin{assumption}
\label{assu:V_nu}We assume the following.
\begin{enumerate}
\item \label{enu:4(a)-is-stable}$\V$ is stable under bounces, i.e. for
all $(x,v)\in\E$ and $k\in\{1,\ldots,K\}$, $v-2(v^{\top}\mathrm{n}_{k}(x))\mathrm{n}_{k}(x)\in\V,$
where $\mathrm{n}_{k}(x)$ is defined by (\ref{eq:nkx});
\item \label{enu:4(b)For-any-}For any $A\subseteq\V$ Borel measurable,
$x\in\X$, we have $\nu((\Id-2\mathrm{n}_{k}(x)\mathrm{n}_{k}(x)^{\top})A)=\nu(A)$,
for any $k\in\{1,\ldots,K\}$.
\item \label{enu:4(c)For-any-bounded}For any bounded and measurable function
$g:\R^{2}\to\R$, $i,j\in\{1,\ldots,d\}$ such that $i\neq j$, $\int_{\V}g(v_{i},v_{j})\,\dif\nu(v)=\int_{\V}g(-v_{i},v_{j})\,\dif\nu(v)$;
\item \label{enu:4(d)-has-finite}$\nu$ has finite fourth order marginal
moment,
\[
m_{4}:=(1/3)\| v_{1}^{2}\|_{2}^{2}=(1/3)\int_{\V}v_{1}^{4}\:\dif\nu(v)<\infty,
\]
and for any $i,j,k,l\in\{1,2,\dots,d\}$ such that $\int_{\V}v_{i}v_{j}v_{k}v_{l}\:\dif\nu(v)=0$
whenever $\mathrm{card}(\{i,j,k,l\})>2$.
\item \label{enu:4(e)Assume-that-.}Assume that $m_{2}\ge1$.
\end{enumerate}
The last condition is purely technical and allows for simpler expressions
in Theorem \ref{thm:abstracthypo}.
\end{assumption}
\noindent We note that this assumption precludes the use of heavy-tailed
distributions for the velocity component $v$. By the discussion after
\cite[H4]{Andrieu18} if $\nu$ is rotation invariant then Assumption
\ref{assu:V_nu}-\ref{enu:4(a)-is-stable}-\ref{enu:4(b)For-any-}-\ref{enu:4(c)For-any-bounded}
are satisfied.  Assumption \ref{assu:V_nu}-\ref{enu:4(b)For-any-}
above implies also that
\[
m_{2,2}:=\| v_{1}v_{2}\|_{2}^{2}=\int_{\V}v_{1}^{2}v_{2}^{2}\:\dif\nu(v)<\infty.
\]

\begin{assumption}
\label{assu:refreshment}The refreshment mechanism is given by
\[
\Rv=\Pi_{v}-\Id.
\]
\end{assumption}
\begin{assumption}
The refreshment rate $\lambda_{\refr}:\X\to\R$ is bounded from below
and above as follows: there exist $\underline{\lambda}>0,c_{\lambda}\ge0$
such that for each $x\in\X$, \label{assu:The-refreshment-rate}
\[
0<\underline{\lambda}\le\lambda_{\refr}(x)\le\underline{\lambda}(1+c_{\lambda}|\nabla_{x}U(x)|).
\]
\end{assumption}

\section{Abstract result}

\label{sec:Abstract-result}We will decompose our operator $\mathcal{L}$
into symmetric and antisymmetric parts, $\mathcal{L}=\mathcal{S}+\mathcal{T}$
on $\C$, as in Condition \ref{cond:basic-core}, where 
\begin{equation}
\mathcal{S}:=(\mathcal{L}+\mathcal{L}^{*})/2,\quad\T:=(\L-\L^{*})/2,\quad\D(\S)=\D(\T)=\C.\label{eq:S_T_def}
\end{equation}
We remark that while our upcoming abstract result, Theorem \ref{thm:abstracthypo},
closely resembles Theorem 2.1 of \cite{Grothaus19}, our definitions
of the abstract operators $\S,\T$ are given by (\ref{eq:S_T_def})
above, which follows the approach of \cite{Andrieu18} instead. In
our PDMP setting we require this approach in order to explicitly identify
the operators $\S,\T$ in Section \ref{sec:Application-to-PDMPs}
and define intermediate quantities and their properties below. By
contrast, in the diffusion setting of \cite{Grothaus19}, the authors
are able to employ It\textroundcap{o}'s Formula to identify their
corresponding symmetric and antisymmetric operators. Thus we cannot
simply use the approach and Theorem 2.1 of \cite{Grothaus19} directly,
but we combine the two approaches of \cite{Andrieu18,Grothaus19}.
\begin{condition}
$\Pi_{v}\C\subset\C$. \label{cond:closable} 
\end{condition}
\noindent As in \cite{Andrieu18} we note that under Conditions \ref{cond:basic-core}
and \ref{cond:closable}, $\T\Pi_{v}$ is closable, with closure $\left(\overline{\T\Pi_{v}},\D\big(\overline{\T\Pi_{v}}\big)\right)$;
this follows from the fact that $\T$ is antisymmetric and $\C$ is
dense as shown in the proof of Lemma \ref{lem:weakPoincaretoAPoincare}.
This allows us to define, by \cite[Theorem 5.1.9]{Pedersen}, 
\[
\A:=\left(m_{2}\mathrm{Id}+(\T\Pi_{v})^{*}(\overline{\T\Pi_{v}})\right)^{-1}(-\T\Pi_{v})^{*},\quad\D(\A)=\D\big((\T\Pi_{v})^{*}\big).
\]
As detailed in \cite[Lemma 3]{Andrieu18} $\A$ is closable with bounded
closure (on $\ELL(\mu)$), and we denote its closure by $\overline{\A}$
hereafter. This is different from \cite{Grothaus14sti,Grothaus19}
where instead it is assumed that $\T$ is closed, or closable, leading
to a definition of $\A$ either involving $\T\Pi_{v}$, or $\bar{\T}\Pi_{v}$,
and for which we could not establish key intermediate results, given
the level of current understanding of PDMPs of the type considered
in this manuscript.
\begin{condition}
\label{cond:proj}We have:
\begin{enumerate}
\item \label{enu:3(a)}$\mathrm{Ran}(\Pi_{v})\subset\mathrm{Ker}(\S^{*})$.
\item \label{enu:3(b)For-any-,}For any $f\in\C$, $\Pi_{v}\T\Pi_{v}f=0$.
\end{enumerate}
\end{condition}
\begin{condition}
\label{cond:R0-rem}There exists some $R_{0}\ge1$ such that for any
$f\in\C$,
\[
\left|\langle\bar{\mathcal{A}}\mathcal{T}(\mathrm{Id}-\Pi_{v})f,f\rangle_{2}+\langle\bar{\mathcal{A}}\mathcal{S}f,f\rangle_{2}\right|\le R_{0}\|(\mathrm{Id}-\Pi_{v})f\|_{2}\|\Pi_{v}f\|_{2}.
\]
\end{condition}
\noindent We state the assumptions on the functional $\Psi$ which
will appear in our weak Poincaré inequalities.
\begin{condition}
\label{cond:Psi}We have a functional $\Psi:\ELL(\mu)\to[0,\infty]$
such that the set $\{f\in\ELL(\mu):\Psi(f)<\infty\}$ is dense in
$\ELL(\mu)$. For any $f\in\D(\L)$ there exists a sequence $\{f_{n}\}_{n=1}^{\infty}\subset\C$
such that $f_{n}\to f$ in $\ELL(\mu)$ and 
\begin{equation}
\limsup_{n\to\infty}\langle-\L f_{n},f_{n}\rangle_{2}\le\langle-\L f,f\rangle_{2},\quad\limsup_{n\to\infty}\Psi(f_{n})\le\Psi(f).\label{eq:Psi_limsup}
\end{equation}
Further, setting $G:=(\T\Pi_{v})^{*}\overline{\T\Pi_{v}}$ with $\D(G)=\{f\in\D(\overline{\T\Pi_{v}}):\overline{\T\Pi_{v}}f\in\D((\T\Pi_{v})^{*})\}$,
we also assume that $\Psi$ satisfies for each $f\in\ELL(\mu)$ and
$t\geq0$, 
\begin{equation}
\Psi(P_{t}f)\le\Psi(f),\Psi(\mathrm{e}^{-tG}f)\le\Psi(f)\text{ and }\Psi(\Pi_{v}f)\le\Psi(f).\label{eq:Psi_ineqs}
\end{equation}
\noindent We now state the required weak Poincaré inequalities.
\end{condition}
\begin{condition}
Assume we have the following weak Poincaré inequalities: for some
decreasing functions $\alpha_{1},\alpha_{2}:(0,\infty)\to[1,\infty)$,\label{cond:weak_poincare}
\begin{equation}
\|\Pi_{v}f-\mu(f)\|_{2}^{2}\le\alpha_{1}(r)\|\T\Pi_{v}f\|_{2}^{2}+r\Psi(\Pi_{v}f),\quad\forall f\in\D(\T\Pi_{v}),r>0,\label{eq:wp_1}
\end{equation}
\begin{equation}
\|(\Id-\Pi_{v})f\|_{2}^{2}\le\alpha_{2}(r)\langle-\S f,f\rangle_{2}+r\Psi(f),\quad\forall f\in\C,r>0.\label{eq:wp_2}
\end{equation}
\end{condition}
The following abstract result is inspired by Theorem 2.1 of \cite{Grothaus19}
and Theorem 4 of \cite{Andrieu18}.
\begin{thm}
\label{thm:abstracthypo}Assume Conditions \ref{cond:basic-core},
\ref{cond:closable}, \ref{cond:proj}, \ref{cond:R0-rem}, \ref{cond:Psi},
\ref{cond:weak_poincare}. Then there exist constants $c_{1},c_{2}>0$
such that
\begin{equation}
\| P_{t}f-\mu(f)\|_{2}^{2}\le\xi(t)\left(\| f\|_{2}^{2}+\Psi(f)\right),\quad\forall t\ge0,f\in\D(\L),\label{eq:Pt_exp}
\end{equation}
for 
\begin{equation}
\xi(t):=c_{1}\inf\left\{ r>0:c_{2}t\ge\alpha_{1}(r)^{2}\alpha_{2}\left(\frac{r}{\alpha_{1}(r)^{2}}\right)\log\frac{1}{r}\right\} .\label{eq:xi_t}
\end{equation}
Expressions for $c_{1}$ and $c_{2}$ are given in (\ref{eq:expression_c1_c2}). 
\end{thm}
\begin{cor}
\label{cor:abstract}Assume the same conditions as Theorem \ref{thm:abstracthypo},
except that (\ref{eq:wp_2}) is replaced by a strong Poincaré inequality,
\begin{equation}
\|(\Id-\Pi_{v})f\|_{2}^{2}\le C_{\mathrm{P}}\langle-\S f,f\rangle_{2},\quad\forall f\in\C,\label{eq:strongPI}
\end{equation}
and assume furthermore that for each $f\in\ELL(\mu)$ with $\Psi(f)<\infty$,
we can find a sequence $(f_{n})\subset\D(\L)$ with
\begin{equation}
f_{n}\to f\text{ in }\ELL(\mu),\quad\limsup_{n\to\infty}\Psi(f_{n})\le\Psi(f).\label{eq:cor_f_n_seq}
\end{equation}
Then we have that (\ref{eq:Pt_exp}) holds, additionally, for any
$f\in\ELL(\mu)$, with 
\[
\xi(t):=c_{1}\inf\left\{ r>0:c'_{2}t\ge\alpha_{1}(r)^{2}\log1/r\right\} ,
\]
for some $c_{1},c'_{2}>0$.
\end{cor}
As we shall see in Example \ref{exa:CLT}, our application of Theorem
\ref{thm:abstracthypo} to PDMPs, given in Theorem \ref{thm:PDMP},
greatly broadens the class of PDMP Monte Carlo processes for which
a central limit theorem holds. The following corollary will be applied
to our examples in Section \ref{sec:Examples}.
\begin{cor}
\label{cor:general-clt}Whenever
\[
\int_{0}^{\infty}t^{-3/2}\xi^{1/2}(t)\,\dif t<\infty,
\]
then for any $f\in\ELL(\mu)$ such that $\mu(f)=0$ and $\Psi(f)<\infty$,
the finite-dimensional distributions of the rescaled process 
\[
t\mapsto\sigma^{-1}N^{-1/2}\int_{0}^{Nt}f(X_{s},V_{s})\,\dif s,
\]
converges as $N\to\infty$ to those of a standard one-dimensional
Brownian motion, where $\sigma>0$ is an appropriately chosen constant
defined in \cite[Theorem MW]{Toth2013}.
\end{cor}
\begin{proof}[Proof of Corollary \ref{cor:abstract}.]
Fix some $f\in\ELL(\mu)$. If $\Psi(f)=\infty$ then (\ref{eq:Pt_exp})
vacuously holds. So assume $\Psi(f)<\infty$, and choose a sequence
$(f_{n})\subset\D(\L)$ satisfying (\ref{eq:cor_f_n_seq}). We can
apply (\ref{eq:Pt_exp}) to each $f_{n}$, and by taking the $\limsup$
we conclude the inequality (\ref{eq:Pt_exp}) is valid for $f$ also.
The alternative expression for $\xi$ is immediate from the expression
(\ref{eq:xi_t}) since in this case, $\alpha_{2}$ can be uniformly
bounded from above. If $c_{2}$ denotes the constant in (\ref{eq:xi_t}),
then we can choose
\[
c'_{2}=\frac{c_{2}}{C_{\mathrm{P}}}.
\]
\end{proof}
\begin{proof}[Proof of Corollary \ref{cor:general-clt}]
 This is a direct application of \cite[Theorem MW]{Toth2013} which
holds whenever, with $f\in\ELL(\mu)$ such that $\int f\,\dif\mu=0$
and $v_{t}:=\int_{0}^{t}P_{s}f\,\dif s$ for each $t>0$,
\begin{equation}
\int_{0}^{\infty}t^{-3/2}\| v_{t}\|_{2}\,\dif t<\infty.\label{eq:CLT_condition}
\end{equation}
\end{proof}
\begin{rem}
The proof of Theorem \ref{thm:abstracthypo} follows the proof of
Theorem 2.1 of \cite{Grothaus19}, however we have adapted the proof
to take into account our differing assumptions and have been careful
to track the constants involved. Due to the structure of the generators
of the PDMPs we consider, under Assumptions \ref{assu:V_nu} and \ref{assu:refreshment},
the operator $\mathcal{S}$ will satisfy (\ref{eq:wp_2}) with $\alpha_{2}(r)$
constant and hence we may set $\Psi=0$ in this inequality and obtain
(\ref{eq:strongPI}); see Corollary \ref{cor:abstract}. This stems
from the specific refreshment mechanism employed here -- the measure
$\nu$ may be required to satisfy a weak Poincaré inequality of the
form of (\ref{eq:weakPI}) when using a diffusion for this update;
see Section \ref{subsec:Checking-ConditionWPI} for a similar situation
involving SDEs. Therefore we include the details of the proof of Theorem
\ref{thm:abstracthypo}, making it straightforward to see how the
constants simplify in Corollary \ref{cor:abstract}. 
\end{rem}
Before we prove Theorem \ref{thm:abstracthypo} we need the following
two lemmas. The following is taken from \cite[Lemma 2.3]{Grothaus19}:
\begin{lem}[{\cite[Lemma 2.3]{Grothaus19}}]
\label{lem:functionofPoincareinequality}Let $(A,\D(A))$ be a densely
defined closed linear operator on a separable Hilbert space $(\mathsf{H},\langle\cdot,\cdot\rangle,\lVert\cdot\rVert).$
Let $(T_{t})_{t\geq0}$ be the $C_{0}$-contraction semigroup generated
by the self-adjoint operator $-A^{*}A$ with domain $\D(A^{*}A):=\{f\in\D(A):Af\in\D(A^{*})\}.$
If the weak Poincaré inequality
\begin{equation}
\lVert f\rVert^{2}\leq\alpha(r)\lVert Af\rVert^{2}+r\varPsi(f),\qquad r>0,f\in\D(A)\label{eq:weak-poincare}
\end{equation}
holds for some decreasing $\alpha:(0,\infty)\to(0,\infty),$ where
$\varPsi:\mathsf{H}\to[0,\infty]$ satisfies
\[
\varPsi(T_{t}f)\leq\varPsi(f),\qquad t\geq0,f\in\D(A).
\]
Then, for any $m_{2}>0$,
\[
\lVert f\rVert^{2}\leq(m_{2}+\alpha(r))\langle(m_{2}\Id+A^{*}A)^{-1}A^{*}Af,f\rangle+r\varPsi(f),\quad r>0,f\in\D(A).
\]
\end{lem}
The following is a consequence of Lemma \ref{lem:functionofPoincareinequality}.
\begin{lem}
\label{lem:weakPoincaretoAPoincare}Assume Conditions \ref{cond:basic-core},
\ref{cond:closable}, \ref{cond:proj}, \ref{cond:R0-rem}, \ref{cond:Psi}
\ref{cond:weak_poincare}, then for any $f\in\C$
\[
\langle\Ab\overline{\T\Pi_{v}}\Pi_{v}f,\Pi_{v}f\rangle_{2}\leq-\frac{1}{m_{2}+\alpha_{1}(r)}\lVert\Pi_{v}f\rVert_{2}^{2}+\frac{r}{m_{2}+\alpha_{1}(r)}\Psi(f).
\]
\end{lem}
\begin{proof}
We apply Lemma \ref{lem:functionofPoincareinequality} with $\mathsf{H}:=\ELL_{0}(\pi)=\left\{ f\in L^{2}(\pi):\pi(f)=0\right\} $
and $(A,\D(A))=\big(((\T\Pi_{v})^{*}\overline{\T\Pi_{v}})^{1/2}\rvert_{\mathsf{H}},\D(\overline{\T\Pi_{v}})\big)$.
The existence of $A$ follows from: (a) $\T$ is anti-symmetric and
$\C$ dense, implying that $\T\Pi_{v}$ is closable of closure we
denote $\big(\overline{\T\Pi_{v}},\D\big(\overline{\T\Pi_{v}}\big)\big)$,
(b) $(\T\Pi_{v})^{*}\overline{\T\Pi_{v}}$ is well defined and self-adjoint
by \cite[Theorem 5.1.9]{Pedersen} and (c) $(A,\D(A))$ exists as
defined and is self-adjoint and positive \cite[Theorem VIII.32]{reed2012methods}.
Note that (\ref{eq:weak-poincare}) holds by (\ref{eq:wp_1}) and
the fact that $\D(\T\Pi_{v})$ is a core for $(\overline{\T\Pi_{v}},\D(\overline{\T\Pi_{v}}))$.
Now since $\D(\A)=\D\big((\T\Pi_{v})^{*}\big)$, for $f\in\D\big((\T\Pi_{v})^{*}(\overline{\T\Pi_{v}})\big)$,
\[
\Ab\overline{\T\Pi_{v}}f=-\left(m_{2}\mathrm{Id}+(\T\Pi_{v})^{*}(\overline{\T\Pi_{v}})\right)^{-1}(\T\Pi_{v})^{*}(\overline{\T\Pi_{v}})f=-\phi\big((\T\Pi_{v})^{*}(\overline{\T\Pi_{v}})\big)f,
\]
where $\phi(s):=s/(m_{2}+s)$. Further from \cite[Theorem 5.1.9]{Pedersen}
$\D\big((\T\Pi_{v})^{*}(\overline{\T\Pi_{v}})\big)$ is a core for
$\D(\overline{\T\Pi_{v}})$ and $\Ab\overline{\T\Pi_{v}}$ is bounded,
implying that on $\D(\overline{\T\Pi_{v}})$, $\Ab\overline{\T\Pi_{v}}=-\phi\big((\T\Pi_{v})^{*}(\overline{\T\Pi_{v}})\big)$.
This together with Condition \ref{cond:Psi} allows us to apply Lemma
\ref{lem:functionofPoincareinequality} to conclude that
\[
\lVert f\rVert_{2}^{2}\leq(m_{2}+\alpha_{1}(r))\langle\phi\big((\T\Pi_{v})^{*}(\overline{\T\Pi_{v}})\big)f,f\rangle_{2}+r\Psi(f),\,\forall r>0,f\in\D(\overline{\T\Pi_{v}}).
\]
 Rearranging, setting $f=\Pi_{v}g$ for $g\in\C\subset\D(\overline{\T\Pi_{v}})$
and using Condition \ref{cond:Psi} we obtain
\[
\langle\Pi_{v}g,\Pi_{v}\T^{*}\A^{*}g\rangle_{2}\leq-\frac{1}{m_{2}+\alpha_{1}(r)}\lVert\Pi_{v}g\rVert_{2}^{2}+\frac{r}{m_{2}+\alpha_{1}(r)}\Psi(g).
\]
\end{proof}
We now establish the following result which is based on \cite[Lemma 5]{Andrieu18}.
To prove the above theorem we need to define the closure $\Ab$ of
$\A$ which is defined on the whole space $\ELL(\mu)$, this is possible
since $\A$ is a bounded operator see \cite[Lemma 3]{Andrieu18}.
Define for any $g\in\D(\L)$
\[
\F_{1}(g)=\langle\L g,g\rangle_{2},\quad\F_{2}(g)=\langle\L g,\Ab g\rangle_{2},\quad,\F_{3}(g)=\langle\Ab\L g,g\rangle_{2}.
\]

\begin{lem}
\label{lem:Fineqs}Assume that Conditions \ref{cond:basic-core},
\ref{cond:closable}, \ref{cond:proj}, \ref{cond:R0-rem}, \ref{cond:Psi},
\ref{cond:weak_poincare} hold. Then for any $g\in\D(\L)$ we have
for any $r_{1},r_{2}>0$,
\begin{align}
\F_{1}(g) & \leq-\frac{1}{\alpha_{2}(r_{2})}\|(\Id-\Pi_{v})g\|_{2}^{2}+\frac{r_{2}}{\alpha_{2}(r_{2})}\Psi(g),\qquad\F_{2}(g)\leq\lVert(\Id-\Pi_{v})g\rVert_{2}^{2},\label{eq:Finequalities}\\
\F_{3}(g) & \leq-\frac{1}{m_{2}+\alpha_{1}(r_{1})}\lVert\Pi_{v}g\rVert_{2}^{2}+\frac{r_{1}}{m_{2}+\alpha_{1}(r_{1})}\Psi(g)+R_{0}\lVert(\Id-\Pi_{v})g\rVert_{2}\lVert\Pi_{v}g\rVert_{2}.\nonumber 
\end{align}
\end{lem}
\begin{proof}
Note that the proof of the inequality for $\F_{2}$ in \cite[Lemma 5]{Andrieu18}
does not rely upon the Poincaré inequality so we may use the same
proof to obtain for all $g\in\D(\L)$,
\[
\F_{2}(g)\leq\lVert(\Id-\Pi_{v})g\rVert_{2}^{2}.
\]
Fix $g\in\C$ and using that $\T$ is antisymmetric we have that $\langle\L g,g\rangle_{2}=\langle\S g,g\rangle_{2}$
and by (\ref{eq:wp_2}) we have for $r>0$,
\[
\F_{1}(g)\leq-\frac{1}{\alpha_{2}(r)}\|(\Id-\Pi_{v})g\|_{2}^{2}+\frac{r}{\alpha_{2}(r)}\Psi(g).
\]
To extend this to $\D(\L)$ fix $f\in\D(\L)$ and let $\left\{ f_{n}\right\} _{n}\subseteq\C$
be as in Condition \ref{cond:Psi}, then we have 
\begin{align*}
\F_{1}(f) & \le\limsup_{n\to\infty}\F_{1}(f_{n})\\
 & \leq\limsup_{n\to\infty}\left(-\frac{1}{\alpha_{2}(r)}\|(\Id-\Pi_{v})f_{n}\|_{2}^{2}+\frac{r}{\alpha_{2}(r)}\Psi(f_{n})\right).
\end{align*}
Now since $\Id-\Pi_{v}$ is bounded and $\limsup_{n\to\infty}\Psi(f_{n})\leq\Psi(f)$
we obtain
\begin{align*}
\F_{1}(f) & \leq-\frac{1}{\alpha_{2}(r)}\|(\Id-\Pi_{v})f\|_{2}^{2}+\frac{r}{\alpha_{2}(r)}\Psi(f).
\end{align*}
Now consider $\F_{3}$, for any $g\in\C\subseteq\D(\L)\cap\D(\L^{*})\cap\D(\T\Pi_{v})$
we have by Condition~\ref{cond:R0-rem},
\begin{align*}
\F_{3}(g) & =\langle\Ab\T\Pi_{v}g,g\rangle_{2}+\langle\Ab\T(\Id-\Pi_{v})g,g\rangle_{2}+\langle\Ab\S g,g\rangle_{2}\\
 & \leq\langle\Pi_{v}g,\Pi_{v}\T^{*}\A^{*}g\rangle_{2}+R_{0}\lVert(\Id-\Pi_{v})g\rVert_{2}\lVert\Pi_{v}g\rVert_{2}.
\end{align*}
This inequality can be extended to $g\in\D(\L)$ since $\C$ is dense
in $\D(\L)$ and $\bar{\mathcal{A}}\overline{\T\Pi_{v}}$ is bounded,
and $\Pi_{v}$ and $\Id-\Pi_{v}$ are bounded. From Lemma \ref{lem:weakPoincaretoAPoincare},
we have
\begin{equation}
\langle\Pi_{v}g,\Pi_{v}\T^{*}\A^{*}g\rangle_{2}\leq-\frac{1}{m_{2}+\alpha_{1}(r)}\lVert\Pi_{v}g\rVert_{2}^{2}+\frac{r}{m_{2}+\alpha_{1}(r)}\Psi(g),\label{eq:F_3_ineq}
\end{equation}
for $g\in\C\subseteq\D(\overline{\T\Pi_{v}})$. Note that (\ref{eq:F_3_ineq})
can be extended to $f\in\D(\L)$: fix $f\in\D(\L)$ and let $\left\{ f_{n}\right\} _{n}\subseteq\C$
be as in Condition~\ref{cond:Psi}. Then (\ref{eq:F_3_ineq}) holds
for each $f_{n}$ and can be extended to $f$, since $\bar{\mathcal{A}}\overline{\T\Pi_{v}}$
is bounded and by (\ref{eq:Psi_limsup}). Finally therefore, putting
the pieces together we can conclude that for each $g\in\D(\L)$,
\[
\F_{3}(g)\leq-\frac{1}{m_{2}+\alpha_{1}(r)}\lVert\Pi_{v}g\rVert_{2}^{2}+\frac{r}{m_{2}+\alpha_{1}(r)}\Psi(g)+R_{0}\lVert(\Id-\Pi_{v})g\rVert_{2}\lVert\Pi_{v}g\rVert_{2}.
\]

\end{proof}
\begin{proof}[Proof of Theorem \ref{thm:abstracthypo}]
 We combine approach of \cite[Theorem 4]{Andrieu18} with that of
\cite[Theorem 2.1]{Grothaus19}. Without loss of generality we may
assume $\mu(f)=0.$ Let us define for any $\epsilon>0$ and $g\in\ELL(\mu)$,
\[
\mathscr{H}_{\epsilon}(g):=\frac{1}{2}\| g\|_{2}^{2}+\epsilon\langle g,\bar{\mathcal{A}}g\rangle.
\]
As in \cite{Andrieu18}, we have the equivalence, for any $0<\epsilon<(m_{2}/2)^{1/2}$
and $g\in\ELL(\mu)$, 
\begin{equation}
\frac{1-(m_{2}/2)^{-1/2}\epsilon}{2}\| g\|_{2}^{2}\le\mathscr{H}_{\epsilon}(g)\le\frac{1+(m_{2}/2)^{-1/2}\epsilon}{2}\| g\|_{2}^{2}.\label{eq:equiv_norms}
\end{equation}
For $f\in\ELL(\mu)$ let us write for convenience, $f_{t}:=P_{t}f$
for each $t\ge0$. Then from the Dynkin formula we know that $f_{t}\in\D(\L)$
and $\dif f_{t}/\dif t=\L f_{t}$ for each $t>0$. Then, we can use
Lemma \ref{lem:Fineqs} to obtain
\begin{align*}
-\frac{\dif}{\dif t}\mathscr{H}_{\epsilon}(f_{t}) & =-[\mathscr{F}_{1}(f_{t})+\epsilon\{\mathscr{F}_{2}(f_{t})+\mathscr{F}_{3}(f_{t})\}]\\
 & \ge\left(\frac{1}{\alpha_{2}(r_{2})}-\epsilon\right)\|(\Id-\Pi_{v})f_{t}\|_{2}^{2}+\frac{\epsilon}{m_{2}+\alpha_{1}(r_{1})}\|\Pi_{v}f_{t}\|_{2}^{2}\\
 & -\left(\frac{r_{2}}{\alpha_{2}(r_{2})}+\frac{\epsilon r_{1}}{m_{2}+\alpha_{1}(r_{1})}\right)\Psi(f_{t})-\epsilon R_{0}\|(\Id-\Pi_{v})f_{t}\|_{2}\|\Pi_{v}f_{t}\|_{2}.
\end{align*}
We now follow the calculations in the proof of \cite[Theorem 2.1]{Grothaus19}.
Our approach is very similar, however we obtain slightly better bounds
(from our Lemma \ref{lem:Fineqs}) which lead to slightly better constants
in the end; hence we include a full proof. We use Young's inequality
to bound the cross term 
\[
\epsilon R_{0}\|(\Id-\Pi_{v})f_{t}\|_{2}\|\Pi_{v}f_{t}\|_{2}\le\frac{\epsilon\|\Pi_{v}f_{t}\|_{2}^{2}}{2(\alpha_{1}(r_{1})+m_{2})}+\frac{\epsilon R_{0}^{2}(\alpha_{1}(r_{1})+m_{2})}{2}\|(\Id-\Pi_{v})f_{t}\|_{2}^{2}.
\]
This gives us
\begin{align*}
-\frac{\dif}{\dif t}\mathscr{H}_{\epsilon}(f_{t})\ge & \left(\frac{1}{\alpha_{2}(r_{2})}-\frac{\epsilon R_{0}^{2}(\alpha_{1}(r_{1})+m_{2})}{2}-\epsilon\right)\|(\Id-\Pi_{v})f_{t}\|_{2}^{2}\\
 & +\frac{\epsilon}{2(\alpha_{1}(r_{1})+m_{2})}\|\Pi_{v}f_{t}\|_{2}^{2}-\left(\frac{r_{2}}{\alpha_{2}(r_{2})}+\frac{\epsilon r_{1}}{m_{2}+\alpha_{1}(r_{1})}\right)\Psi(f_{t}).
\end{align*}
Now we take 
\begin{equation}
\epsilon=\frac{1}{\alpha_{2}(r_{2})[R_{0}^{2}(\alpha_{1}(r_{1})+m_{2})+2]}<\frac{1}{2},\label{eq:epsilon}
\end{equation}
and using the fact that $\| f_{t}\|_{2}^{2}=\|\Pi_{v}f_{t}\|_{2}^{2}+\|(\Id-\Pi_{v})f_{t}\|_{2}^{2}$
and (\ref{eq:Psi_ineqs}) we get
\begin{align*}
-\frac{\dif}{\dif t}\mathscr{H}_{\epsilon}(f_{t})\ge & \frac{1}{2\alpha_{2}(r_{2})}\|(\Id-\Pi_{v})f_{t}\|_{2}^{2}\\
 & +\frac{1}{2\alpha_{2}(r_{2})(\alpha_{1}(r_{1})+m_{2})[R_{0}^{2}(\alpha_{1}(r_{1})+m_{2})+2]}\|\Pi_{v}f_{t}\|_{2}^{2}\\
 & -\left(\frac{r_{2}}{\alpha_{2}(r_{2})}+\frac{r_{1}}{(m_{2}+\alpha_{1}(r_{1}))\alpha_{2}(r_{2})[R_{0}^{2}(\alpha_{1}(r_{1})+m_{2}]+2)}\right)\Psi(f_{t})\\
\ge & \frac{1}{4R_{0}^{2}\alpha_{2}(r_{2})(\alpha_{1}(r_{1})+m_{2})^{2}}\| f_{t}\|_{2}^{2}\\
 & -\left(\frac{r_{2}}{\alpha_{2}(r_{2})}+\frac{r_{1}}{R_{0}^{2}\alpha_{2}(r_{2})(\alpha_{1}(r_{1})+m_{2})^{2}}\right)\Psi(f_{t}).
\end{align*}
Now since $\epsilon<1/2$ we use (\ref{eq:equiv_norms}) to see that
for $g\in\ELL(\mu)$
\[
\| g\|_{2}^{2}\ge\frac{4\mathscr{H}_{\epsilon}(g)}{2+(m_{2}/2)^{-1/2}},
\]
so 
\begin{align*}
-\frac{\dif}{\dif t}\mathscr{H}_{\epsilon}(f_{t})\ge & \frac{\mathscr{H}_{\epsilon}(f_{t})}{(2+(m_{2}/2)^{-1/2})R_{0}^{2}\alpha_{2}(r_{2})(\alpha_{1}(r_{1})+m_{2})^{2}}\\
 & -\left(\frac{r_{2}}{\alpha_{2}(r_{2})}+\frac{r_{1}}{R_{0}^{2}\alpha_{2}(r_{2})(\alpha_{1}(r_{1})+m_{2})^{2}}\right)\Psi(f_{t}).
\end{align*}
So by Gronwall's lemma, for $t\geq0$,
\begin{align*}
\mathscr{H}_{\epsilon}(f_{t})\le & \exp\left[-\frac{t}{(2+(m_{2}/2)^{-1/2})R_{0}^{2}\alpha_{2}(r_{2})(\alpha_{1}(r_{1})+m_{2})^{2}}\right]\mathscr{H}_{\epsilon}(f)\\
 & +\left[r_{2}(2+(m_{2}/2)^{-1/2})R_{0}^{2}(\alpha_{1}(r_{1})+m_{2})^{2}+r_{1}(2+(m_{2}/2)^{-1/2})\right]\Psi(f).
\end{align*}
Now we choose $r_{1}=r$, $r_{2}=r/\alpha_{1}(r_{1})^{2}$, and then
using that $m_{2}\ge1$ (Assumption~\ref{assu:V_nu}) and (\ref{eq:equiv_norms}),
\[
\| f_{t}\|_{2}^{2}\le c_{1}\exp\left[-\frac{c_{2}t}{\alpha_{1}(r)^{2}\alpha_{2}\left(r/\alpha_{1}(r)^{2}\right)}\right]\| f\|_{2}^{2}+c_{1}r\Psi(f),
\]
for $r>0,$ $f\in\D(\L)$, $t\ge0$.  Here we can take for $\epsilon$
as defined in (\ref{eq:epsilon}),
\begin{align}
c_{1} & =\frac{2\max\left\{ \frac{1+\epsilon(m_{2}/2)^{-1/2}}{2},(2+(m_{2}/2)^{-1/2})\left[R_{0}^{2}(1+m_{2})^{2}+1\right]\right\} }{1-\epsilon(m_{2}/2)^{-1/2}},\nonumber \\
c_{2} & =\frac{1}{R_{0}^{2}(2+(m_{2}/2)^{-1/2})(1+m_{2})^{2}}.\label{eq:expression_c1_c2}
\end{align}
Thus we can conclude that (\ref{eq:Pt_exp}) holds with $\xi$ as
in (\ref{eq:xi_t}) for some $c_{1},c_{2}>0$.%
\end{proof}

\section{Application to PDMPs}

\label{sec:Application-to-PDMPs}

In this case the operator $\L$ acts on smooth functions in $\C=\mathrm{C}_{b}^{2}(\E)$
as follows,
\begin{equation}
\L f(x,v)=v^{\top}\nabla_{x}f(x,v)+\sum_{k=1}^{K}\lambda_{k}(x,v)(\bounce_{k}-\Id)f(x,v)+m_{2}^{1/2}\lambda_{\refr}(x)\Rv f(x,v).\label{eq:LPDMPdef}
\end{equation}
In which case we have that the operators $\S$ and $\T$ are given
for functions $f\in\C$ by
\begin{align*}
\S f(x,v) & =\frac{1}{2}\sum_{k=1}^{K}\lambda_{k}^{\mathrm{e}}(x,v)(\bounce_{k}-\Id)f(x,v)+m_{2}^{1/2}\lambda_{\refr}(x)\Rv f(x,v),\\
\T f(x,v) & =v^{\top}\nabla_{x}f(x,v)+\frac{1}{2}\sum_{k=1}^{K}v^{\top}F_{k}(\bounce_{k}-\Id)f(x,v).
\end{align*}
Where, $\lambda_{k}^{\mathrm{e}}(x,v):=\lambda_{k}(x,v)+\lambda_{k}(x,-v)$,
for $k\in\{1,\dots,K\}$ and $(x,v)\in\E$. Recall that $G:=(\T\Pi_{v})^{*}\overline{\T\Pi_{v}}$
with $\D(G)=\{f\in\D(\T\Pi_{v}):\T\Pi_{v}f\in\D((\T\Pi_{v})^{*})\}$.
By \cite[Lemma 9(b)]{Andrieu18}, we have that $\mathrm{C}_{b}^{2}(\E)\subset\D(G)$
and for any $f\in\mathrm{C}_{b}^{2}(\E)$, $Gf=m_{2}\nabla_{x}^{*}\nabla_{x}\Pi_{v}f$,
here $\nabla^{*}$ is the adjoint of $\nabla$ as an operator from
$\D(\nabla)\subseteq\ELL(\pi)\to\mathrm{L}_{\pi}^{2}(\X;\R^{d})$.
We take $\Psi=\|\cdot\|_{\osc}^{2}$. Our goal in the section is the
following theorem, which will follow from Corollary \ref{cor:abstract}
once we have checked that the abstract conditions hold.
\begin{thm}
\label{thm:PDMP} Assume that our basic Condition \ref{cond:basic-core}
and our PDMP Assumptions \ref{assu:U}, \ref{assu:F_k}, \ref{assu:event_rate},
\ref{assu:V_nu}, \ref{assu:refreshment}, \ref{assu:The-refreshment-rate}
hold. Then we have convergence of the semigroup,
\[
\| P_{t}f\|_{2}^{2}\le\xi(t)\left(\| f\|_{2}^{2}+\| f\|_{\osc}^{2}\right),\quad\forall t\ge0,f\in\ELL(\mu),
\]
where 
\[
\xi(t):=c_{1}\inf\left\{ r>0:c_{2}t\ge\alpha_{1}(r)^{2}\log1/r\right\} ,
\]
for a decreasing function $\alpha_{1}:(0,\infty)\to[1,\infty)$ and
some constants $c_{1},c_{2}>0$. In particular, $c_{2}$ may be taken
to be
\[
c_{2}=\frac{\underline{\lambda}m_{2}^{1/2}}{R_{0}^{2}(2+(m_{2}/2)^{-1/2})(1+m_{2})^{2}}.
\]
\end{thm}
\begin{example}[Central limit theorems.]
\label{exa:CLT}From our results in Example \ref{exa:5}, we can
apply Corollary \ref{cor:general-clt} to our running examples to
see that a central limit theorem holds for $f\in\ELL(\mu)$ such that
$\| f\|_{{\rm osc}}<\infty$,
\begin{enumerate}
\item for $U(x)=\frac{1}{2}(d+p)\log\left(1+|x|^{2}\right)$ whenever $p$
is large enough so that $\tau<1/2$, where $\tau$ is defined later
in (\ref{eq:alphaforpowertarget}),
\item for $U(x)=\sigma|x|{}^{\delta}$ for any $\sigma>0$, $0<\delta<1$.
\end{enumerate}
\end{example}

\subsection{Checking Condition \ref{cond:basic-core}}

In \cite{Andrieu18} it is argued that the BPS and the ZZ processes
are both well-defined Markov processes satisfying Condition \ref{cond:basic-core}
with $\C=\ctb(\E)$ as a core (see their remarks after their Corollary
2). In order to help the reader we provide here a brief overview of
existing theoretical results which have been used to establish a similar
property, and can be adapted to establish Condition \ref{cond:basic-core}.
For the BPS, it is shown in \cite{Durmus2018}, in full detail, that
$\mathrm{C}_{\mathrm{c}}^{1}(\E)$ is core for its generator on $\mathrm{C}_{0}(\E)$,
the set of continuous functions vanishing at infinity. This relies
on a stability property of $\mathrm{C}^{1}(\E)$, the set of continuously
differentiable functions, under the semigroup \cite[Lemma 17]{Durmus2018},
which by using \cite[Proposition 3.3]{ethier2009markov} implies the
core property. As remarked in \cite[Remark 18]{Durmus2018}, \cite[Lemma 17]{Durmus2018}
can be extended to cover stability of $\mathrm{C}_{\mathrm{c}}^{k}(\E)$
for $k\geq2$ and an application of \cite[Proposition 3.3]{ethier2009markov}
leads to the desired conclusion for the generator on $\mathrm{C}_{0}(\E)$.
Crucially \cite[Lemma 17]{Durmus2018} requires the intensity to belong
to $\mathrm{C}^{1}(\E)$, a property not satisfied by the standard
BPS or ZZ when using the ``canonical'' choice of intensity. This
is however relaxed for BPS by utilising their \cite[Theorem 21]{Durmus2018}
in conjunction with \cite[Propositions 9 and 23]{Durmus2018}. In
\cite[Theorem 5.11]{Andrieu19} these ideas are used in the context
of the ZZ process to establish in detail that $\mathrm{C}_{\mathrm{c}}^{1}(\E)$
is a core for the generator on $\ELL(\mu)$, assuming that the intensities
involved belong to $\mathrm{C}^{1}(\E)$--this extends directly to
the scenario where $\mathrm{C}^{1}(\E)$ is replaced with $\mathrm{C}^{2}(\E)$.
We note that \cite[Proposition 5.17]{Andrieu19} defines a family
of smooth intensities, $\mathrm{C}^{2}(\E)$ under Condition \ref{assu:U},
uniformly converging to the popular canonical choice, so we may apply
\cite[Proposition 11 and 27]{Durmus2018} to get that $\mathrm{C}_{\mathrm{c}}^{2}(\E)$
is a core for the generator in $\mathrm{C}_{0}(\E)$ of ZZ with the
canonical intensity. Since the semigroup preserves the domain of the
generator in $\mathrm{C}_{0}(\E)$, when we extend the semigroup to
$\ELL(\mu)$ we have that the domain of the generator is a core for
the generator in $\ELL(\mu)$ by using \cite[Proposition 3.3]{ethier2009markov}.
Then since convergence in $\mathrm{C}_{0}(\E)$ implies convergence
in $\ELL(\mu)$ we have that $\mathrm{C}_{\mathrm{c}}^{2}(\E)$ is
a core for the generator of ZZ with the canonical intensity in $\ELL(\mu)$.
This immediately implies that $\ctb(\E)$ is a core in $\ELL(\mu)$
as desired. See also \cite{Bierkens2019Spec} for a direct approach
in the one-dimensional ZZ case.

\subsection{Checking Condition \ref{cond:closable} and \ref{cond:proj}}

Condition \ref{cond:closable} follows from Lemma 9 of \cite{Andrieu18}.
Condition \ref{cond:proj}-\ref{enu:3(a)} is immediate from the definition
of $\S$. Finally, Condition \ref{cond:proj}-\ref{enu:3(b)For-any-,}
follows from Lemma 9 of \cite{Andrieu18}.

\subsection{Checking Condition \ref{cond:weak_poincare}: weak Poincaré inequalities
\label{subsec:Checking-ConditionWPI}}

To establish weak Poincaré inequalities, our starting point is \cite{Rockner01},
as in \cite{Grothaus19}. \cite[Theorem 3.1]{Rockner01} and the subsequent
remark allow us to deduce that there exists decreasing functions $\alpha_{1},\alpha_{2}:(0,\infty)\to[1,\infty)$
such that the weak Poincaré inequalities hold
\begin{align}
\pi(f^{2})-\pi(f)^{2} & \le\alpha_{1}(r)\pi(|\nabla_{x}f|^{2})+r\| f\|_{\osc}^{2},\forall f\in\mathrm{C}_{\mathrm{b}}^{1}(\X),r>0.\label{eq:weakPI}
\end{align}
We need to show conditions (\ref{eq:wp_1}) and (\ref{eq:strongPI})
hold. First by \cite[Lemma 9 (b)]{Andrieu18} we have for any $f\in\C$,
\[
(\T\Pi_{v})^{*}(\T\Pi_{v})f=m_{2}\nabla_{x}^{*}\nabla_{x}\Pi_{v}f.
\]
Multiplying by $\Pi_{v}f$ and integrating we obtain
\[
\lVert(\T\Pi_{v})f\rVert_{2}^{2}=m_{2}\lVert\nabla\Pi_{v}f\rVert_{2}^{2}.
\]
Therefore substituting the above expression into (\ref{eq:weakPI})
we have
\[
\lVert\Pi_{v}f-\mu(f)\rVert_{2}^{2}\le\frac{\alpha_{1}(r)}{m_{2}}\lVert\T\Pi_{v}f\rVert_{2}^{2}+r\| f\|_{\osc}^{2},\forall f\in\mathrm{C}_{\mathrm{b}}^{1}(\E),r>0.
\]
Thus we have (\ref{eq:wp_1}). Note that (\ref{eq:strongPI}) follows
immediately from \cite[Proposition 10]{Andrieu18}, since we have
the same refreshment mechanism. Indeed, we can obtain for $f\in\C$,
\[
\|(\Id-\Pi_{v})f\|_{2}^{2}\le\frac{1}{\underline{\lambda}m_{2}^{1/2}}\langle-\S f,f\rangle.
\]

\begin{example}[Example \ref{exa:targets} continued]
\label{exa:targetscont} By \cite[Theorem 3.2]{Rockner01} we have
that (\ref{eq:weakPI}) holds and 
\[
\alpha(r)=\frac{4R_{r}^{2}}{\pi^{2}}e^{\delta_{R_{r}}(U)}
\]
where $R_{r}:=\inf\{s>0:\pi(B_{s}^{c})\leq r/(1+r)\}$, $\delta_{R}(U)=\sup_{\{x,y\in B_{R}\}}U(x)-U(y)$
and $B_{R}=\{x\in\R^{d}:|x|\leq R\}$. We shall exhibit $\alpha$
for the two cases considered in Example \ref{exa:targets}
\begin{enumerate}
\item If $U(x)=\frac{1}{2}(d+p)\log\left(1+|x|^{2}\right)$, for some $p>0$
then by \cite[Example 1.4(a)]{Rockner01} we have that (\ref{eq:weakPI})
holds with
\begin{align}
\alpha_{1}(r) & =c(1+r^{-\tau})\nonumber \\
\tau= & \min\left\{ \frac{d+p+2}{p},\frac{4p+4+2d}{[p^{2}-4-2d-2p]^{+}}\right\} ,\label{eq:alphaforpowertarget}
\end{align}
where $c$ is a universal constant independent of dimension and $r$.
Note that \cite[Example 1.4(a)]{Rockner01} considers a slightly different
function, $V(x)=(d+p)\log\left(1+|x|\right)$. However, the difference
between these functions is bounded so if (\ref{eq:weakPI}) holds
for $V$, then it also holds for $U$.
\item If $U(x)=\sigma|x|{}^{\delta}$ then by \cite[Example 1.4 (c)]{Rockner01}
we have that (\ref{eq:weakPI}) holds with 
\begin{equation}
\alpha_{1}(r)=c\left[1+\log(1+r^{-1})\right]^{4(1-\delta)/\delta}.\label{eq:alphaforsubexpo}
\end{equation}
\end{enumerate}
We detail convergence rates these potentials lead to in Section \ref{sec:Examples}.
\end{example}

\subsection{Checking Condition \ref{cond:R0-rem}: finding $R_{0}$}

The most difficult part of the proof is checking Condition \ref{cond:R0-rem}
which will control the remainder terms. For light-tailed targets in
\cite{Andrieu18} this is done by showing that solutions of the Poisson
equation
\begin{equation}
m_{2}(\mathrm{Id}+\nabla_{x}^{*}\nabla_{x})u_{f}=\Pi_{v}f,\label{eq:Poisson-eq}
\end{equation}
have polynomially growing derivatives. This implies in their setting
that they are $\pi$-integrable, however for heavy-tailed measures
$\pi$ this is not sufficient. By using Schauder estimates we know
that the solution $u_{f}$ of the Poisson equation is twice differentiable,
however we do not know in general that the derivatives are $\pi$-integrable.
By multiplying the solution by smooth cut-off functions it is shown
in \cite{Lorenzi06} that under Assumption \ref{assu:U} the first
and second derivatives are in $\ELL(\pi)$. We can write down the
solution as
\[
u_{f}=m_{2}^{-1}(\mathrm{Id}+\nabla_{x}^{*}\nabla_{x})^{-1}\Pi_{v}f,
\]
since $\nabla_{x}$ is a densely-defined closed operator on $\ELL(\pi)$,
\cite[Proposition 26]{Andrieu18} shows that $(\mathrm{Id}+\nabla_{x}^{*}\nabla_{x})^{-1}$
is a positive self-adjoint bounded operator on $\ELL(\pi)$, which
furthermore is a bijection between $\ELL(\pi)$ and $\mathrm{D}(\nabla_{x}^{*}\nabla_{x})$.
In our case, we will utilize the powerful abstract result of \cite{Lorenzi06}.

We also remark here that our subsequent argument patches a minor omission
in \cite{Grothaus19}, in their proof of (H3) for degenerate diffusions.
In \cite{Grothaus19}, the authors reference \cite{Dolbeault15} and
\cite[Section 5.1]{Grothaus14sti} for elliptic a priori estimates.
However, the cited references assume the existence of a strong Poincaré
inequality, which precisely falls outside the scope of the processes
under consideration.
\begin{lem}
\label{lem:Poissonequestimates}Under our Assumption \ref{assu:U}
for any $f\in\ctb(\mathsf{E})$,
\begin{enumerate}
\item the solution $u_{f}$ of the Poisson equation (\ref{eq:Poisson-eq})
is uniquely defined and $u_{f}\in\mathsf{\mathrm{C}}^{3}(\X)$;
\item there exist some $\kappa_{1},\kappa_{2}>0$ such that: 
\begin{equation}
\max\left\{ \lVert u_{f}\rVert_{2},\lVert\nabla_{x}u_{f}\rVert_{2}\right\} \leq m_{2}^{-1}\lVert\Pi_{v}f\rVert_{2},\label{eq:derPE}
\end{equation}
\begin{align}
\|\nabla_{x}^{2}u_{f}\|_{2} & \le m_{2}^{-1}\kappa_{1}\|\Pi_{v}f\|_{2},\label{eq:secondderPE}
\end{align}
\begin{equation}
\lVert(\nabla_{x}U)^{\top}\nabla_{x}u_{f}\rVert_{2}\leq m_{2}^{-1}\kappa_{2}\lVert\Pi_{v}f\rVert_{2}.\label{eq:innerprodderPE}
\end{equation}
Here 
\[
\kappa_{1}=\sqrt{2(2+c_{U})},
\]
and if $|\nabla_{x}U|$ is bounded then we may take $\kappa_{2}=\sup_{x\in\X}\lvert\nabla_{x}U(x)\rvert$,
otherwise $\kappa_{2}=\sqrt{4(4\kappa_{1}+C_{U}d^{1+\omega})}$, where
$\omega,c_{U}$ and $C_{U}$ are as in Assumption~\ref{assu:U}.
\end{enumerate}
\end{lem}
\begin{proof}
Note by rescaling $f$ we may take $m_{2}=1.$ The fact that $u_{f}\in\ELL(\pi)$
follows from the fact that $(\mathrm{Id}+\nabla_{x}^{*}\nabla_{x})^{-1}$
is a positive self-adjoint bounded operator on $\ELL(\pi)$, as detailed
in \cite[Proposition 26]{Andrieu18}. We now make use of \cite[Theorem 3.3]{Lorenzi06}.
Since we are dealing with a simplified version of the Poisson equation
(\ref{eq:Poisson-eq}), the Hypotheses 2.1(i)-(iii) of \cite{Lorenzi06}
are trivially satisfied. Hypothesis 2.1(iv) of \cite{Lorenzi06} is
equivalent in our setting to Assumption \ref{assu:U}-\ref{enu:HessianUlowerbounded}.
Hence we have satisfied the hypotheses of \cite[Theorem 3.3]{Lorenzi06}.
The bounds (\ref{eq:derPE}) then follow immediately from \cite[Theorem 3.3]{Lorenzi06}.
The finiteness and upper bound of $\|\nabla_{x}^{2}u_{f}\|_{2}$ follows
from revisiting the proof of \cite[Theorem 3.3]{Lorenzi06}, see Appendix
\ref{appen:pf_lemma}. If there exists $\kappa_{2}$ such that $\lVert\nabla_{x}U\rVert_{\infty}\leq\kappa_{2}$
then (\ref{eq:innerprodderPE}) follows from (\ref{eq:derPE}). Now
consider the case where (\ref{eq:boundonU}) holds. 

In the following, for functions $\R^{d}\to\R^{d}$, we will use the
bare notation $\|\cdot\|$ to denote the norm $\|\cdot\|_{\mathrm{L}_{\pi}^{2}(\R^{d};\R^{d})}$.
Following the proof of \cite[Lemma 34]{Andrieu18} we have for any
$\phi\in\mathrm{C}_{\mathrm{c}}^{\infty}(\X)$ and $\varepsilon>0$
that 
\[
\lVert\phi\nabla_{x}U\lVert^{2}-\langle\phi^{2},\Delta_{x}U\rangle_{2}\leq\varepsilon^{-1}\lVert\nabla_{x}\phi\rVert^{2}+\varepsilon\lVert\phi\nabla_{x}U\rVert^{2}.
\]
Now using (\ref{eq:boundonU}) we obtain
\[
\frac{1}{2}\lVert\phi\nabla_{x}U\lVert^{2}-C_{U}d^{1+\omega}\lVert\phi\rVert_{2}^{2}\leq\varepsilon^{-1}\lVert\nabla_{x}\phi\rVert^{2}+\varepsilon\lVert\phi\nabla_{x}U\rVert^{2}.
\]
Rearranging and setting $\varepsilon=1/4$ gives
\begin{equation}
\frac{1}{4}\lVert\phi\nabla_{x}U\lVert^{2}\leq4\lVert\nabla_{x}\phi\rVert^{2}+C_{U}d^{1+\omega}\lVert\phi\rVert^{2}.\label{eq:boundforphinablaU}
\end{equation}
As $\mathrm{C}_{\mathrm{c}}^{\infty}$ is a core for $(\nabla_{x},\D(\nabla_{x}))$
we have the above inequality for any $\phi\in\D(\nabla_{x})$. In
particular we shall set $\phi(x)=\partial_{i}u_{f}$ which gives
\[
\frac{1}{4}\lVert\partial_{i}u_{f}\nabla_{x}U\lVert^{2}\leq4\lVert\nabla_{x}\partial_{i}u_{f}\rVert^{2}+C_{U}d^{1+\omega}\lVert\partial_{i}u_{f}\rVert^{2}.
\]
Now summing over $i\in\{1,\ldots,d\}$ we obtain 
\begin{align*}
\lVert\,\rvert\nabla_{x}u_{f} & \lvert\,\nabla_{x}U\rVert^{2}=\sum_{i=1}^{d}\lVert\partial_{i}u_{f}\nabla_{x}U\lVert^{2}\\
\leq & 4(4\lVert\nabla_{x}^{2}u_{f}\rVert_{\mathrm{L}_{\pi}^{2}(\R^{d};\R^{d\times d})}^{2}+C_{U}d^{1+\omega}\lVert\nabla_{x}u_{f}\rVert^{2})\\
\leq & 4(4\kappa_{1}+C_{U}d^{1+\omega})\lVert\Pi_{v}f\rVert_{2}^{2}.
\end{align*}
Recall here that $\lvert\nabla_{x}u_{f}(x)\rvert$ denotes the Euclidean
norm of $\nabla_{x}u_{f}(x)$. Finally by using Cauchy--Schwarz we
obtain the desired result with $\kappa_{2}=\sqrt{4(4\kappa_{1}+C_{U}d^{1+\omega})}$.
\end{proof}
A proof of the following result is given in \cite[Lemma 9(c)]{Andrieu18};
however it relies on a density argument involving $\mathsf{\mathrm{C}}_{{\rm poly}}^{3}(\X)$,
therefore requiring the existence of moments under $\pi$ and hence
stronger assumption on $U$. The below establishes that this assumption
is not required.
\begin{lem}
Under Assumption \ref{assu:U}, then for any $f\in\ELL(\pi)$
\begin{equation}
\{m_{2}\mathrm{Id}+(\T\Pi_{v})^{*}\overline{\T\Pi_{v}}\}^{-1}f=m_{2}\{\mathrm{Id}+\nabla_{x}^{*}\nabla_{x}\Pi_{v}\}^{-1}f.\label{eq:AisTPibutalsoNabla}
\end{equation}
\end{lem}
\begin{proof}
The proof is along the same lines as that of \cite[Lemma 9(c)]{Andrieu18},
replacing $\mathsf{\mathrm{C}}_{{\rm poly}}^{3}(\X)$ with $\mathsf{\mathrm{H}}^{2}(\X)$,
thanks to the results of \cite{Lorenzi06}. More precisely from \cite[Lemma 3.1]{Lorenzi06}
we have that for any $g\in\mathrm{H}^{2}(\X)$ we can find $\{g_{n}\in\mathsf{\mathrm{C}}_{{\rm b}}^{2}(\X)\}_{n\in\mathbb{N}}$
such that $\nabla_{x}g_{n}\rightarrow\nabla_{x}g$ and $\nabla_{x}^{2}g_{n}\rightarrow\nabla_{x}^{2}g$
in $\ELL(\pi)$, therefore implying $\{m_{2}\mathrm{Id}+(\T\Pi_{v})^{*}\overline{\T\Pi_{v}}\}g_{n}\rightarrow m_{2}\{\mathrm{Id}+\nabla_{x}^{*}\nabla_{x}\}g$
and $\{m_{2}\mathrm{Id}+(\T\Pi_{v})^{*}\overline{\T\Pi_{v}}\}g_{n}\rightarrow m_{2}\{\mathrm{Id}+\nabla_{x}^{*}\nabla_{x}\}^{-1}g$
since $\nabla_{x}^{*}\nabla_{x}h=-\Delta_{x}h+\nabla_{x}U^{\top}\nabla_{x}h$
for $h\in\mathrm{H}^{2}(\mathsf{X})$ (see Lemma \ref{lem:nabla-star})
and the two operators are closed. Therefore the two operators coincide
on $\mathrm{H}^{2}(\X)\subset\D\big(m_{2}\mathrm{Id}+(\T\Pi_{v})^{*}\overline{\T\Pi_{v}}\big)\cap\D\big(\mathrm{Id}+\nabla_{x}^{*}\nabla_{x}\big)$.
Again from \cite{Lorenzi06} we have $\mathsf{\mathrm{C}}_{{\rm b}}^{2}(\X)\subset m_{2}(\mathrm{Id}+\nabla_{x}^{*}\nabla_{x})\big(\mathrm{H}^{2}(\X)\big)$
(see Lemma \ref{lem:Poissonequestimates}), which is dense in $\ELL(\pi)$
and we deduce (\ref{eq:AisTPibutalsoNabla}) by boundedness of the
two inverses.
\end{proof}
In order to show that Condition \ref{cond:R0-rem} follows we use
\cite[Lemma 12 \& 13]{Andrieu18} which states the following.
\begin{lem}[{\cite[Lemma 12 \& 13]{Andrieu18}}]
\label{lem:Remaindertoellipticity}Assume that $\L$ is given by
(\ref{eq:LPDMPdef}). Assume in addition that Assumptions \ref{assu:U}
- \ref{assu:The-refreshment-rate} hold.
\begin{enumerate}
\item For any $f\in\ctb(\mathsf{E})$,
\[
\lvert\langle\A\S(\Id-\Pi_{v})f,f\rangle_{2}\rvert\leq\lVert(\Id-\Pi_{v})f\rVert_{2}\lVert(\Id-\Pi_{v})\S\A^{*}f\rVert_{2}.
\]
\item For any $f\in\ctb(\mathsf{E}),$ 
\begin{align*}
\lVert(\Id-\Pi_{v})\S\A^{*}f\rVert_{2} & \leq m_{2}(\lVert\lambda_{\refr}\nabla_{x}u_{f}\rVert_{2}+c_{\varphi}K\|\nabla_{x}u_{f}\|_{2})\\
 & +C_{\varphi}\sqrt{2m_{2,2}+3(m_{4}-m_{2,2})_{+}}\sum_{k=1}^{K}\lVert F_{k}^{\top}\nabla_{x}u_{f}\rVert_{2}.
\end{align*}
\item For any $f\in\ctb(\mathsf{E})$,
\[
\lvert\langle\A\T(\Id-\Pi_{v})f,f\rangle_{2}\rvert\leq\lVert(\Id-\Pi_{v})f\rVert_{2}\lVert(\Id-\Pi_{v})\T\A^{*}f\rVert_{2}.
\]
\item For any $f\in\ctb(\mathsf{E})$,
\[
\lVert(\Id-\Pi_{v})\T\A^{*}f\rVert_{2}\leq\sqrt{3(m_{4}-m_{2,2})+2m_{2,2}}(\kappa_{1}+\kappa_{2})\lVert\Pi_{v}f\rVert_{2},
\]
with $\kappa_{1}$ and $\kappa_{2}$ positive constants as defined
in Lemma \ref{lem:Poissonequestimates}.
\end{enumerate}
\end{lem}
\begin{cor}
\label{cor:R0-PDMP}Assume that $\L$ is given by (\ref{eq:LPDMPdef}).
Assume in addition that Assumptions \ref{assu:U} - \ref{assu:The-refreshment-rate}
hold. Then Condition \ref{cond:R0-rem} holds with
\begin{align*}
R_{0} & =\sqrt{3(m_{4}-m_{2,2})_{+}+2m_{2,2}}(\kappa_{1}+\kappa_{2})\\
 & +m_{2}^{-1}\left(m_{2}(\underline{\lambda}(1+c_{\lambda}\kappa_{2})+c_{\varphi}K)+C_{\varphi}(1+\kappa_{2})\sqrt{2m_{2,2}+3(m_{4}-m_{2,2})_{+}}\sum_{k=1}^{K}a_{k}\right).
\end{align*}
\end{cor}
\begin{proof}[Proof of Corollary \ref{cor:R0-PDMP} ]
 By Lemma \ref{lem:Remaindertoellipticity}, we have
\begin{align*}
\left|\langle\bar{\mathcal{A}}\mathcal{T}(\mathrm{Id}-\Pi_{v})f,f\rangle_{2}+\langle\bar{\mathcal{A}}\mathcal{S}f,f\rangle_{2}\right| & \leq\\
\sqrt{3(m_{4}-m_{2,2})_{+}+2m_{2,2}} & (\kappa_{1}+\kappa_{2})\lVert\Pi_{v}f\rVert_{2}\lVert(1-\Pi_{v})f\rVert_{2}\\
+\biggl( & C_{\varphi}\sqrt{2m_{2,2}+3(m_{4}-m_{2,2})_{+}}\sum_{k=1}^{K}\lVert F_{k}^{\top}\nabla_{x}u_{f}\rVert_{2}\\
+m_{2} & (\lVert\lambda_{\refr}\nabla_{x}u_{f}\rVert_{2}+c_{\varphi}K\|\nabla_{x}u_{f}\|_{2})\biggr)\lVert(1-\Pi_{v})f\rVert_{2}.
\end{align*}
Now by Assumption \ref{assu:F_k}-\ref{enu:2(c)for-all-,} and Lemma
\ref{lem:Poissonequestimates} we have
\[
\lVert F_{k}^{\top}\nabla_{x}u_{f}\rVert_{2}\leq a_{k}(\lVert\nabla_{x}u_{f}\rVert_{2}+\lVert\lvert\nabla_{x}U\rvert\nabla_{x}u_{f}\rVert_{2})\leq m_{2}^{-1}a_{k}(1+\kappa_{2})\lVert\Pi_{v}f\rVert_{2}.
\]
Similarly using Assumption \ref{assu:The-refreshment-rate} we have
\[
\lVert\lambda_{\refr}\nabla_{x}u_{f}\rVert_{2}\leq\underline{\lambda}\lVert(1+c_{\lambda}|\nabla_{x}U(x)|)\nabla_{x}u_{f}\rVert_{2}\leq m_{2}^{-1}\underline{\lambda}(1+c_{\lambda}\kappa_{2})\lVert\Pi_{v}f\rVert_{2}.
\]
Therefore Condition \ref{cond:R0-rem} follows with
\begin{align*}
R_{0} & =\sqrt{3(m_{4}-m_{2,2})+2m_{2,2}}(\kappa_{1}+\kappa_{2})\\
+ & m_{2}^{-1}\left(m_{2}(\underline{\lambda}(1+c_{\lambda}\kappa_{2})+c_{\varphi}K)+C_{\varphi}(1+\kappa_{2})\sqrt{2m_{2,2}+3(m_{4}-m_{2,2})_{+}}\sum_{k=1}^{K}a_{k}\right).
\end{align*}
\end{proof}

\subsection{Checking Condition \ref{cond:Psi}}

Recall that here $\Psi(\cdot)=\|\cdot\|_{\osc}^{2}$. We combine the
approaches of \cite{Grothaus19} and \cite{Andrieu18}. Note then
that the conditions in (\ref{eq:Psi_ineqs}) follow immediately from
the definition of $\Psi$ and the contractivity in $\mathrm{L}^{\infty}(\mu)$
of the Markov semigroups $P_{t}$ and $\mathrm{\mathrm{e}}^{-tG}$;
the latter corresponds to a diffusion semigroup. Now we check (\ref{eq:Psi_limsup}).
Fix some $f\in\D(\L)$ with $\Psi(f)=\| f\|_{\osc}^{2}<\infty$. Without
loss of generality, by translating $f$, we can assume that $\mu(f)=0$.
Hence we have $\gamma_{1}:=\text{ess}_{\mu}\inf f\le0$ and $\gamma_{2}:=\text{ess}_{\mu}\sup f\ge0$.
Since $\C=\ctb(\E)$ is a core, we can choose a sequence $\{g_{n}\}_{n=1}^{\infty}\subset\C$
such that $g_{n}\to f$ and $\L g_{n}\to\L f$ in $\ELL(\mu)$. We
take, as in \cite{Grothaus19}, for each $n\in\mathbb{N}$ a monotone
increasing function $h_{n}\in\mathrm{C}^{\infty}(\R)$ which satisfies
$0\le h'_{n}\le1$ and
\[
h_{n}(r)=\begin{cases}
r & r\in[\gamma_{1},\gamma_{2}],\\
\gamma_{1}-\frac{1}{2n} & r\le\gamma_{1}-\frac{1}{n},\\
\gamma_{2}+\frac{1}{2n} & r\ge\gamma_{2}+\frac{1}{n}.
\end{cases}
\]
Now similarly we set $f_{n}:=h_{n}(g_{n})\in\C$ and we have $f_{n}\to f$
in $\ELL(\mu)$. By construction $\| f_{n}\|_{\osc}\leq\gamma_{2}-\gamma_{1}+\frac{1}{n}$
so we have
\[
\limsup_{n\to\infty}\Psi(f_{n})=\limsup_{n\to\infty}\| f_{n}\|_{\osc}^{2}\leq(\gamma_{2}-\gamma_{1})^{2}=\| f\|_{\osc}^{2}.
\]
Additionally, 
\begin{align*}
\limsup_{n\to\infty}\langle-\L f_{n},f_{n}\rangle & =\limsup_{n\to\infty}\langle-\S f_{n},f_{n}\rangle\\
 & =\limsup_{n\to\infty}\sum_{k=0}^{K}\frac{1}{2}\langle\lambda_{k}^{\mathrm{e}}f_{n},f_{n}-\mathcal{B}_{k}f_{n}\rangle-\langle m_{2}^{1/2}\lambda_{\refr}\Rv f_{n},f_{n}\rangle\\
 & =\limsup_{n\to\infty}\frac{1}{4}\sum_{k=0}^{K}\int\lambda_{k}^{\mathrm{e}}(x,v)\left\{ f_{n}(x,v)-\mathcal{B}_{k}f_{n}(x,v)\right\} ^{2}\dif\mu(x,v)\\
 & \quad+m_{2}^{1/2}\int\lambda_{\refr}(x)\left\{ (\Id-\Pi_{v})f_{n}(x,v)\right\} ^{2}\dif\mu(x,v)\\
 & \le\limsup_{n\to\infty}\frac{1}{4}\sum_{k=0}^{K}\int\lambda_{k}^{\mathrm{e}}(x,v)\left\{ g_{n}(x,v)-\mathcal{B}_{k}g_{n}(x,v)\right\} ^{2}\dif\mu(x,v)\\
 & \quad+m_{2}^{1/2}\int\lambda_{\refr}(x)\left\{ (\Id-\Pi_{v})g_{n}(x)\right\} ^{2}\dif\mu(x,v)\\
 & =\limsup_{n\to\infty}\langle-\S g_{n},g_{n}\rangle\\
 & =\langle-\L f,f\rangle.
\end{align*}
The third equality follows from the fact that $\mathcal{B}_{k}$ is
symmetric on $\ELL(\mu)$, and $\mathcal{B}_{k}\lambda_{k}^{\mathrm{e}}(x,v)=\lambda_{k}^{\mathrm{e}}(x,v)$,
as in the proof of Proposition 7 of \cite{Andrieu19}. The inequality
follows from the fact that $h_{n}$ is 1-Lipschitz; we have that $\left(h_{n}(x)-h_{n}(y)\right)^{2}\le(x-y)^{2}$
for any $x,y\in\R$. Thus we have verified Condition \ref{cond:Psi}.

Since $\C$ is a core for $\L$, and $\L$ is densely defined, $\C$
is also dense in $\ELL(\mu)$. So given some $f\in\ELL(\mu)$ with
$\Psi(f)<\infty$, this argument also allows us to conclude that there
exists a sequence $(f_{n})\subset\D(\L)$ satisfying (\ref{eq:cor_f_n_seq})
as required in the assumptions of Corollary \ref{cor:abstract}.

\section{Examples\label{sec:Examples}}

In this section we apply Theorem \ref{lem:nabla-star} to our running
examples and obtain explicit bounds on convergence rate. We further
explore the tightness of such bounds on various examples, both theoretically
and empirically. Our main finding is that although our bounds are
useful (e.g. we establish the existence of a central limit theorem
for a large class of problems; see Example \ref{exa:CLT}) and widely
applicable, they are not sharp and rather pessimistic. In particular
we find that the bounds we obtain for PDMPs do not compare favourably
with the corresponding bounds for (reversible) Langevin diffusions
for a particular heavy-tailed target density. Informally, this should
not be surprising since for PDMPs, condition (\ref{eq:wp_1}) is precisely
that required of a Langevin diffusion (there is equivalence in the
reversible case) to achieve a particular subgeometric rate of convergence.
This condition drives all subsequent developments where the nonreversible
nature of the initial process does not seem to play a rôle anymore. 
\begin{example}[Example \ref{exa:targets} and \ref{exa:targetscont} continued]
\label{exa:5} In Example \ref{exa:targetscont} we showed the weak
Poincaré Inequality holds for the two examples considered and we now
show what rate we obtain by applying Theorem \ref{thm:PDMP}. These
obtained rates will immediately allow us to check condition (\ref{eq:CLT_condition})
to ensure central limit theorems, as discussed in Example \ref{exa:CLT}.
\end{example}
\begin{itemize}
\item[(a)]  For the case $U(x)=\frac{1}{2}(d+p)\log(1+\lvert x\rvert^{2})$
for some $p>0$, we have from Example \ref{exa:targetscont} (a) that
$\alpha_{1}$ is given by (\ref{eq:alphaforpowertarget}). Hence by
Theorem \ref{thm:PDMP}, we have the bound, for some $c>0$,
\[
\| P_{t}f\|_{2}^{2}\le ct^{-\frac{1}{2\tau}}\left(\| f\|_{2}^{2}+\lVert f\rVert_{\osc}^{2}\right),\quad\forall t\ge0,f\in\D(\L).
\]
\item[(b)]  Consider the case $U(x)=\sigma|x|{}^{\delta}$, for $x\in\R^{d}$
with $|x|\ge M$, for some $\delta,\sigma,M>0$, by Example \ref{exa:targetscont}
(b) we have that $\alpha_{1}$ is given by (\ref{eq:alphaforsubexpo}).
By Theorem \ref{thm:PDMP} we have that (\ref{eq:Pt_exp}) holds with
$\xi(t)=\inf\left\{ r>0:c_{2}t\geq\alpha_{1}(r)^{2}\log(r^{-1})\right\} .$
Setting $r=\exp(-kt^{\frac{\delta}{8-7\delta}})$ we have
\begin{align*}
\alpha_{1}(r)^{2}\log(r^{-1}) & =c^{2}\left[1+\log(1+r^{-1})\right]^{8(1-\delta)/\delta}\log(r^{-1})\\
 & =c^{2}k\left[1+\log(1+\exp(kt^{\frac{\delta}{8-7\delta}}))\right]^{8(1-\delta)/\delta}t^{\frac{\delta}{8-7\delta}}\\
 & \leq c^{2}k\left[1+\exp(-kt^{\frac{\delta}{8-7\delta}})+\log(\exp(kt^{\frac{\delta}{8-7\delta}}))\right]^{8(1-\delta)/\delta}t^{\frac{\delta}{8-7\delta}}
\end{align*}
to obtain the inequality in the last line we use that $\log(1+x)-\log(x)=\log(1+x^{-1})\leq x^{-1}$
for $x\geq1$. Now for $t\leq1$ the required bound is immediate so
we shall assume $t\geq1$ in which case $1+\exp(-kt^{\frac{\delta}{8-7\delta}})\leq2\leq2kt^{\frac{\delta}{8-7\delta}}$
so there exists $C(k,\delta)>0$ such that
\begin{align*}
\alpha_{1}(r)^{2}\log(r^{-1}) & \leq C(k,\delta)\left[t^{\frac{\delta}{8-7\delta}}\right]^{8(1-\delta)/\delta}t^{\frac{\delta}{8-7\delta}}\\
 & =C(k,\delta)t^{\frac{8(1-\delta)}{8-7\delta}}t^{\frac{\delta}{8-7\delta}}=C(k,\delta)t.
\end{align*}
Therefore we have, for some $c>0$,
\[
\| P_{t}f\|^{2}\le\exp(-kt^{\frac{\delta}{8-7\delta}})\left(\| f\|^{2}+\lVert f\rVert_{\osc}^{2}\right),\quad\forall t\ge0,f\in\D(\L).
\]
Let us compare the rates we obtain with those found for the reversible
and nonreversible Langevin diffusion. In \cite{Rockner01} the authors
consider the reversible (overdamped) Langevin diffusion,
\[
\dif X_{t}=-\nabla_{x}U(X_{t})\,\dif t+\sqrt{2}\,\dif B_{t},
\]
for heavy-tailed target distributions and prove convergence to equilibrium
by using the weak Poincaré Inequality and standard techniques, whereas
in \cite{Grothaus19} they use hypocoercivity to prove convergence
to equilibria for the nonreversible (underdamped) Langevin diffusion,
\begin{align*}
\dif X_{t} & =\nabla_{v}V_{2}(V_{t})\,\dif t,\\
\dif V_{t} & =-[\nabla_{x}V_{1}(X_{t})+\nabla_{v}V_{2}(V_{t})]\,\dif t+\sqrt{2}\,\dif B_{t}.
\end{align*}
In this case the diffusion has a unique invariant measure with density
which is proportional to $e^{-V_{1}(x)-V_{2}(v)}.$ In \cite[Example 1.4]{Rockner01}
and \cite[Example 1.3]{Grothaus19} they find the rates of convergence
to equilibria which we summarise in the table below. We can see in
Table 1 that for the examples we have considered we obtain the same
rate of convergence for PDMP to those obtained in \cite{Grothaus19}
for non-reversible Langevin dynamics which have Gaussian velocity.
This is a demonstration of the limits of hypocoercivity theory for
subgeometric target distributions, since the rate $\xi(t)$ given
by Theorem \ref{thm:abstracthypo} depends only on $\alpha_{1},\alpha_{2}$
and $c_{2}$ but $\alpha_{1}$ and $\alpha_{2}$ are given by the
target measure so are the same for each algorithm. 
\end{itemize}
\begin{center}
\begin{table}
\begin{centering}
\begin{tabular}{|c|c|c|}
\hline 
$U$ in scenario: & (a) & (b)\tabularnewline
\hline 
\hline 
rate of PDMP & $t^{-\frac{1}{2\tau}}$ & $\exp(-kt^{\frac{\delta}{8-7\delta}})$\tabularnewline
\hline 
rate of reversible Langevin & $t^{-\frac{1}{\tau}}$ & $\exp(-kt^{\frac{\delta}{4-3\delta}})$\tabularnewline
\hline 
nonreversible Langevin with $V_{1}=U$, $V_{2}=v^{2}/2$ & $t^{-\frac{1}{2\tau}}$ & $\exp(-kt^{\frac{\delta}{8-7\delta}})$\tabularnewline
\hline 
nonreversible Langevin with $V_{1}=v^{2}/2$, $V_{2}=U$ & $t^{-\frac{1}{\tau}}$ & $\exp(-kt^{\frac{\delta}{4-3\delta}})$\tabularnewline
\hline 
\end{tabular}\caption{Comparison of rates of convergence for scenarios in Example \ref{exa:targets}-\ref{exa:targetscont}.}
\par\end{centering}
\end{table}
\par\end{center}
\begin{example}
Let $d=1$ and $U(x)=\log(1+x^{2})$, for this choice we can give
the rate explicitly. By \cite[Theorem 3.2]{Rockner01} we have that

\[
\alpha_{1}(r)=\frac{4R_{r}^{2}}{\pi^{2}}\exp(\delta_{R_{r}}(U)).
\]
Here $R_{r}:=\inf\left\{ U(x)-U(y):x,y\in B_{R}\right\} $ and $\delta_{R}(U)=\sup\left\{ U(x)-U(y):x,y\in B_{R}\right\} .$
Now,
\[
\pi(B_{s}^{c})=2\int_{s}^{\infty}\frac{1}{1+x^{2}}dx=\pi-2\arctan(s).
\]
Rearranging, we find
\begin{align*}
R_{r} & =\tan\left(\frac{1}{2}\left(\pi-\frac{r}{r+1}\right)\right),\\
\delta_{R}(U) & =\log(1+R^{2}).
\end{align*}
Substituting these expressions into the defintion of $\alpha_{1}(r)$
we have
\[
\alpha_{1}(r)=\frac{4}{\pi^{2}}\frac{\tan\left(\frac{1}{2}\left(\pi-\frac{r}{r+1}\right)\right)}{\cos\left(\frac{1}{2}\left(\pi-\frac{r}{r+1}\right)\right)^{2}}\sim\frac{32}{\pi^{2}r^{3}}.
\]
Now we will run the Zig-Zag sampler with this potential, in which
case one finds that
\[
c_{2}'=\frac{\lambda_{\text{\ensuremath{\refr}}}}{4R_{0}^{2}(2+\sqrt{2})},\quad R_{0}=(1+\sqrt{1+c_{U}/2})+\lambda_{\refr}+2\sqrt{2},\quad c_{1}=2+\sqrt{2}.
\]
For this choice of $U$ we find $c_{U}=0.3$. Then
\[
\xi(t):=\inf\left\{ r>0:c_{2}'t\geq\frac{16}{\pi^{4}}\frac{\tan\left(\frac{1}{2}\left(\pi-\frac{r}{r+1}\right)\right)^{2}}{\cos\left(\frac{1}{2}\left(\pi-\frac{r}{r+1}\right)\right)^{4}}\log\left(\frac{1}{r}\right)\right\} .
\]
To leading order\footnote{This is found by first approximating to leading order 
\[
\frac{16}{\pi^{4}}\frac{\tan\left(\frac{1}{2}\left(\pi-\frac{r}{r+1}\right)\right)^{2}}{\cos\left(\frac{1}{2}\left(\pi-\frac{r}{r+1}\right)\right)^{4}}\sim\frac{1024}{\pi^{4}r^{6}}
\]

then we can solve

\[
c_{2}'t=\frac{1024}{\pi^{4}r^{6}}\log\left(\frac{1}{r}\right)=\frac{512}{3\pi^{4}r^{6}}\log\left(\frac{1}{r^{6}}\right).
\]
In which case $r^{-6}=W(\frac{3c_{2}'\pi^{4}}{512}t)$, therefore
$\xi(t)\sim W(\frac{3c_{2}'\pi^{4}}{512}t)^{-1/6}$.} $\xi(t)\sim W(\frac{3c_{2}'\pi^{4}}{512}t)^{-1/6}$, here $W(x)$
is the Lambert function defined as the inverse of $xe^{x}.$ We can
also compare this with the numerical performance for the Zig-Zag sampler
with canonical switching rates, below is a plot of $\mathbb{E}[f(X_{t})]^{2}$
started with inital condition $X_{0}=-5$ and with the velocity $V_{0}$
drawn uniformly from $\{1,-1\}$. To run the simulation we generated
$N=10^{7}$ Zig-Zag samplers and then calculated $N^{-1}\sum_{n=1}^{N}f(X_{t}^{n})$
to estimate $\mu(f)$. In the figure below we have used $f(x)=\mathbb{I}\{x\geq5\}$.
As we can see in the plot the process appears to converge much faster
than the theoretical bound of $\xi(t)$ which is included as a reference.
Note that in some of the plots the error is converging to a constant
value (around $10^{-4}$) this is due to error in running a finite
number of particles to estimate the expectation. We have also included
on the plot a simulation for the non-reversible Langevin SDE
\begin{align*}
{\rm d}X_{t} & =V_{t}{\rm d}t\\
{\rm d}V_{t} & =-\nabla_{x}U(X_{t}){\rm d}t-V_{t}{\rm d}t+\sqrt{2}{\rm d}B_{t},
\end{align*}
where $(B_{t})_{t\geq0}$ is a one-dimensional Brownian motion. To
simlute the non-reversible Langevin process we used the Euler-Maryuma
scheme with step size 0.01. We see for this example that the Zig-Zag
process is converging to zero faster than the non-reversible Langevin
SDE. 

\begin{figure}
\centering{}\includegraphics[width=1\textwidth]{"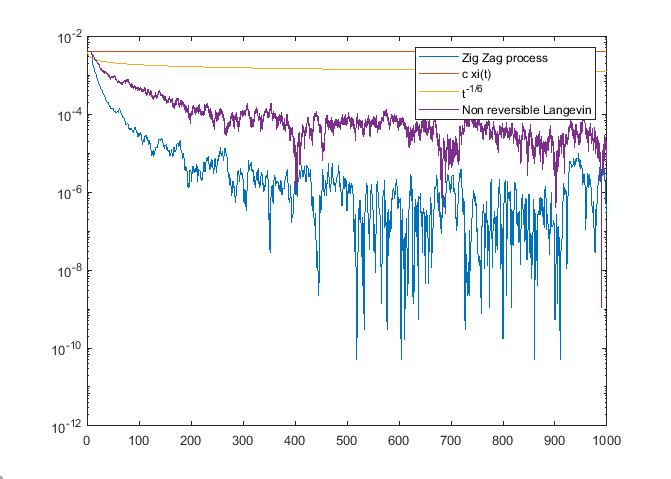"}\caption{A plot of the error $\lvert P_{t}f(-5,v)-\mu(f)\rvert^{2}$ for $f(x,v)=\mathbb{I}\{x\protect\geq5\}$,
in red is a plot of $c\xi(t)$ where $c$ is chosen so that curves
are initially equal. The line in blue is the error from the Zig-Zag
process whereas the purple is an Euler approximation to the non-reversible
Langevin process using a step size of 0.01. }
\end{figure}
\end{example}

\appendix

\section{Proof of Lemma \ref{lem:Poissonequestimates}}
\begin{defn}
A (smooth) cut-off function $\vartheta\colon\X\rightarrow[0,1]$ is
such that
\begin{enumerate}
\item $\vartheta\in\mathsf{\mathrm{C}}_{\mathrm{c}}^{\infty}(\X)$,
\item $\vartheta\equiv0$ on the complement of the unit ball $B_{1}^{\mathrm{\complement}}$
,
\item $\vartheta\equiv1$ on the ball $B_{1/2}$.
\end{enumerate}
\end{defn}
\begin{rem}
\label{rem:cut-off-decreasing-away-origin}Define the mollifier $\omega\in\mathrm{C}_{c}^{\infty}\big(\R^{d}\big)$,
\[
\omega(x)=\begin{cases}
\exp\big(-1/(1-|x|^{2})\big) & \text{if }|x|<1\\
0 & \text{otherwise,}
\end{cases}
\]
then $\vartheta(x)=2^{-d}\int\mathbb{I}\{|y-x|\leq1\}\omega(2y)\,\dif y$
is a cut-off function. We now show that if $|x_{2}|\geq|x_{1}|$ then
$\vartheta(x_{2})\leq\vartheta(x_{1})$. First notice that since $\omega$
is invariant by rotation, so is $\vartheta$, and we can focus on
the following scenario. Let $x\in B_{1/2}^{\complement}$ and $\lambda>1$
then $\{y\in B_{1/2}\colon|y-\lambda x|\leq1\}\subset\{y\in B_{1/2}\colon|y-x|\leq1\}$
and therefore $\vartheta(\lambda x)\leq\vartheta(x)$, from which
we conclude. The inclusion above is justified by the fact that for
$(x,y)\in B_{1/2}^{\complement}\times B_{1/2}$ $\lambda\mapsto|\lambda x-y|^{2}$
is non-decreasing since, with $n(x)=x/|2x|$, 
\begin{align*}
|\lambda x-y|^{2} & =|\lambda x-n(x)|^{2}+|n(x)-y|^{2}+2\langle\lambda x-n(x),n(x)-y\rangle\\
 & =(\lambda|x|-1/2)^{2}+|n(x)-y|^{2}+2(\lambda|x|-1/2)\langle n(x),n(x)-y\rangle,
\end{align*}
and $\langle n(x),n(x)-y\rangle=\langle n(x),n(x)\rangle-\langle n(x),y\rangle\geq0$
as $\langle n(x),y\rangle\leq|n(x)|\,|y|\leq|n(x)|^{2}$.
\end{rem}

\subsection{Expression for $\nabla^{*}$}
\begin{lem}
\label{lem:nabla-star}Let Assumption \ref{assu:U} hold. Then with
$\left(\nabla_{x},\D(\nabla_{x})\right)$ with $\nabla_{x}\colon\D(\nabla_{x})\rightarrow\ELL(\pi)^{d}$
where $\D(\nabla_{x})\subset\ELL(\pi)$, 
\begin{enumerate}
\item for $F\in\mathrm{H}^{1}(\X,\R^{d})$,
\[
\nabla_{x}^{*}F=-{\rm div}_{x}F+\nabla_{x}U^{\top}F,
\]
\item for $f\in\mathsf{\mathrm{H}}^{2}(\X)$,
\[
\nabla_{x}^{*}\nabla_{x}f=-\Delta_{x}f+\nabla_{x}U^{\top}\nabla_{x}f.
\]
\end{enumerate}
\end{lem}
\begin{proof}
We follow the ideas of \cite{Lorenzi06}. Consider first the case
where $d=1$. Let $f\in\mathrm{H}^{1}(\X)$ and $\vartheta\colon\X\rightarrow[0,1]$
a cut-off function, then we define for $n\in\mathbb{N_{*}}$ and $x\in\X$
$f_{n}(x):=f(x)\vartheta(x/n)$ . Note that
\[
\partial f_{n}(x)=\vartheta(x/n)\partial f(x)+\frac{1}{n}f(x)\partial\vartheta(x/n).
\]
Let $f,g\in\mathsf{H}^{1}(\X)$. Then for any $n\in\mathbb{N}_{*}$,
using integration by parts and noting that $\vartheta(\cdot/n)\equiv0$
on $B_{n}^{\complement}$, we have
\begin{align*}
\int f_{n}(x)\pi(x)\partial g_{n}(x)\,{\rm d}x\\
=-Z^{-1}\int g_{n}(x) & \left[\partial f_{n}(x)-f_{n}(x)\partial U(x)\right]\exp(-U(x))\,{\rm d}x\\
=-\int\vartheta^{2}(x/n) & g(x)\bigl[\partial f(x)-f(x)\partial U(x)+f(x)\frac{\partial\vartheta(x/n)}{n}\bigr]\,\pi({\rm d}x).
\end{align*}
Letting $n\rightarrow\infty$ and applying the dominated convergence
theorem leads to the desired result; the required integrability follows
from applying the Cauchy--Schwarz inequality several times, Assumption
\ref{assu:U}--\ref{enu:nablaU-summable} and that $f(x)\partial U(x)\in\ELL(\pi)$.
Note if $\partial U$ is bounded then it is immediate to see $f(x)\partial U(x)\in\ELL(\pi)$
on the other hand if $\partial U$ is not bounded then Assumption
\ref{assu:U}- \ref{enu:cond-laplacian-U} holds and we use (\ref{eq:boundforphinablaU})
with $\varphi=f$ to show $f(x)\partial U(x)\in\ELL(\pi)$, indeed
this gives 
\[
\lVert f\partial U\rVert_{2}^{2}\leq16\lVert\partial f\rVert_{2}^{2}+4C_{U}d^{1+\omega}\lVert f\rVert_{2}^{2}.
\]
 When $d>1$, simply notice that for $f\in\mathrm{H}^{1}(\X)$ and
$G=(g_{1},\ldots,g_{d})\in\mathrm{H}^{1}(\X,\R^{d})$, $\bigl\langle\nabla_{x}f,G\bigr\rangle_{\pi}=\sum_{i=1}^{d}\bigl\langle\partial_{i}f,g_{i}\bigr\rangle_{\pi}$
and apply the result above. The second statement is immediate.
\end{proof}
\begin{cor}
Let $\L$ be as in (\ref{eq:def-calL}) and assume Assumption \ref{assu:U}
holds. Then for $f\in\mathrm{C}_{b}^{2}(\E)$ and $(x,v)\in\E$,
\begin{align*}
\L^{*}f(x,v) & =-v^{\top}\nabla_{x}f(x,v)+\sum_{k=1}^{K}\lambda(x,-v)({\cal B}_{k}-{\rm Id})f(x,v)+m_{2}^{1/2}\lambda_{\refr}(x)\mathcal{R}_{v}f(x).
\end{align*}
\end{cor}
\begin{proof}
From Lemma \ref{lem:nabla-star}, for $F\in\mathrm{H}^{1}(\X,\R^{d})$
and $g\in{\rm H}^{1}(\pi)$ we have $\bigl\langle\nabla_{x}g,F\bigr\rangle_{\pi}=\langle g,-{\rm div}_{x}F+\nabla_{x}U^{\top}F\rangle_{\pi}$.
We need to check that this is applicable in the case $x\mapsto F_{v}(x):=v\cdot f(x,v)$
for fixed $v\in\V$, provided $F_{v}\in\mathrm{H}^{1}(\mathsf{E},\mathbb{R}^{d})$,
which is clearly true here. This is true for $v=0$. For $v\neq0$,
noting that $\nabla_{x}(v\cdot f)=v(\nabla_{x}f)^{\top}$ we deduce
that $F_{v}\in{\rm H}^{1}(\mathsf{X},\R^{d})$. The rest then follows
from a calculation identical to that in the proof of \cite[Proposition 7]{Andrieu18}. 
\end{proof}

\subsection{Bound on $\protect\|\nabla_{x}^{2}u_{f}\protect\|_{2}$ in Lemma
\ref{lem:Poissonequestimates}\label{appen:pf_lemma}}

 Fix $f\in\mathrm{C}_{b}^{2}(\E)$, and let $u=u_{f}$ be the solution
of the Poisson equation (\ref{eq:Poisson-eq}). Since $\mathrm{C}_{\mathrm{c}}^{\infty}(\X)$
is dense in $\mathrm{H}^{2}(\X)$ (see \cite[Lemma 3.1]{Lorenzi06})
we may assume that $f$ is smooth with compact support. By rescaling
$f$ we may assume that $m_{2}=1.$ Differentiating (\ref{eq:Poisson-eq})
we have for $i\in\{1,\dots,d\}$,
\begin{equation}
\partial_{i}u+\nabla_{x}^{*}\nabla_{x}\partial_{i}u+(\nabla_{x}\partial_{i}U)^{\top}\nabla_{x}u=\partial_{i}f,\label{eq:derivativeofPE-1}
\end{equation}
since from Lemma \ref{lem:Poissonequestimates} $u\in\mathrm{C}^{3}(\X)$
as $U\in\mathrm{C}^{2+\alpha}(\X)$. Let $\vartheta\colon\X\rightarrow[0,1]$
be a cut-off function. Define $\vartheta_{n}(x):=\vartheta(x/n)$
for $(n,x)\in\mathbb{N_{*}\times\X}$ and note that $\sup_{x\in B_{n}}|\nabla_{x}\vartheta_{n}(x)|\leq C_{1}n^{-1}$
with $C_{1}:=\sup_{x\in B_{1}}|\nabla_{x}\vartheta(x)|$. Now let
$n\geq n_{0}$ for $n_{0}\in\mathbb{N}_{*}$ such that $\vartheta_{n_{0}}\equiv1$
on the support of $f$. Throughout this section, we will write $\langle\cdot,\cdot\rangle_{\pi}$
as shorthand for $\langle\cdot,\cdot\rangle_{\ELL(\pi)}$. Now multiply
(\ref{eq:derivativeofPE-1}) by $\vartheta_{n}^{2}\partial_{i}u$,
sum over $i\in\{1,\ldots,d\}$ and integrate to get
\begin{multline}
\left\langle \vartheta_{n}^{2}\nabla_{x}u,\nabla_{x}u\right\rangle _{\pi}+\sum_{i=1}^{d}\left(\left\langle \vartheta_{n}^{2}\partial_{i}u,\nabla_{x}^{*}\nabla_{x}\partial_{i}u\right\rangle _{\pi}+\left\langle \vartheta_{n}^{2}\partial_{i}u,(\nabla_{x}\partial_{i}U)^{\top}\nabla_{x}u\right\rangle _{\pi}\right)\\
=\left\langle \vartheta_{n}^{2}\nabla_{x}u,\nabla_{x}f\right\rangle _{\pi}.\label{eq:derivativeofPE}
\end{multline}
Now since $\vartheta_{n}\equiv1$ on the support of $f$ and from
Lemma \ref{lem:nabla-star}, we deduce $\left\langle \vartheta_{n}^{2}\nabla_{x}u,\nabla_{x}f\right\rangle _{\pi}=\left\langle \nabla_{x}^{*}\nabla_{x}u,f\right\rangle _{\pi}$
and therefore
\begin{multline}
\|\vartheta_{n}\nabla_{x}u\|_{2}^{2}+\sum_{i=1}^{d}\left(\left\langle \vartheta_{n}^{2}\partial_{i}u,\nabla_{x}^{*}\nabla_{x}\partial_{i}u\right\rangle _{\pi}+\left\langle \vartheta_{n}^{2}\partial_{i}u,(\nabla_{x}\partial_{i}U)^{\top}\nabla_{x}u\right\rangle _{\pi}\right)\\
=\left\langle \nabla_{x}^{*}\nabla_{x}u,f\right\rangle _{\pi}.\label{eq:3termsPE}
\end{multline}
Now consider the second term on the left hand side of (\ref{eq:3termsPE}),
\begin{align*}
\sum_{i=1}^{d}\left\langle \vartheta_{n}^{2}\partial_{i}u,\nabla_{x}^{*}\nabla_{x}\partial_{i}u\right\rangle _{\pi} & =\sum_{i=1}^{d}\left\langle \nabla_{x}(\vartheta_{n}^{2}\partial_{i}u),\nabla_{x}\partial_{i}u\right\rangle _{\pi}\\
 & =\|\vartheta_{n}\nabla_{x}^{2}u\|_{2}^{2}+2\sum_{i,j=1}^{d}\int\vartheta_{n}\partial_{j}\vartheta_{n}\,\partial_{i}u\,\partial_{ij}u\,\dif\pi.
\end{align*}
Therefore (\ref{eq:3termsPE}) becomes 
\[
\|\vartheta_{n}\nabla_{x}u\|_{2}^{2}+\|\vartheta_{n}\nabla_{x}^{2}u\|_{2}^{2}=I_{1}+I_{2}+I_{3},
\]
with
\begin{gather}
I_{1}:=\left\langle \nabla_{x}^{*}\nabla_{x}u,f\right\rangle _{\pi},I_{2}:=-2\left\langle \vartheta_{n}\nabla_{x}\vartheta_{n},\nabla_{x}^{2}u\nabla_{x}u\right\rangle _{\pi}\nonumber \\
I_{3}:=-\left\langle \vartheta_{n}\nabla u,\nabla_{x}^{2}U\,(\vartheta_{n}\nabla_{x}u)\right\rangle _{\pi}.\label{eq:I1I2I3}
\end{gather}
It remains to estimate each of the terms $I_{1},I_{2},I_{3}$. For
the first term $I_{1}$, recall that $\nabla_{x}^{*}\nabla_{x}u=f-u$
and $\lVert u\rVert_{2}\leq\lVert f\rVert_{2}$ by (\ref{eq:derPE})
so we have
\begin{equation}
\bigl|\left\langle \nabla_{x}^{*}\nabla_{x}u,f\right\rangle _{\pi}\bigr|\leq\lVert\nabla_{x}^{*}\nabla_{x}u\rVert_{2}\lVert f\rVert_{2}\leq\lVert f-u\rVert_{2}\lVert f\rVert_{2}\leq2\lVert f\rVert_{2}^{2}.\label{eq:boundforI1}
\end{equation}
For the third term $I_{3}$ we use that $\nabla_{x}^{2}U\succeq-c_{U}\mathrm{I}_{d}$
by Assumption \ref{assu:U} for some $c_{U}\geq0$, $\vartheta_{n}\leq1$
and that $\lVert\nabla_{x}u\rVert_{2}\leq\lVert f\rVert_{2}$ by (\ref{eq:derPE}),
\begin{equation}
\bigl|\left\langle \vartheta_{n}\nabla_{x}u,\nabla_{x}^{2}U\,(\vartheta_{n}\nabla_{x}u)\right\rangle \bigr|\leq c_{U}\|\vartheta_{n}\nabla_{x}u\|_{2}^{2}\leq c_{U}\lVert f\rVert_{2}^{2}.\label{eq:boundforI3}
\end{equation}
For the second term $I_{2}$, we use that $|\nabla_{x}\vartheta_{n}(x)|\leq C_{1}n^{-1}$
then
\begin{align*}
\bigl|\left\langle \vartheta_{n}\nabla_{x}\vartheta_{n},\nabla_{x}^{2}u\,\nabla_{x}u\right\rangle _{\pi}\bigr| & \leq C_{1}n^{-1}\int\vartheta_{n}\lvert\nabla_{x}^{2}u\rvert\lvert\nabla_{x}u\rvert\,\dif\pi.
\end{align*}
Now Young's inequality for any $\varepsilon>0$ and (\ref{eq:derPE})
yield
\begin{align}
\bigl|\left\langle \vartheta_{n}\nabla_{x}\vartheta_{n},\nabla_{x}^{2}u\,\nabla_{x}u\right\rangle _{\pi}\bigr|\leq & \varepsilon\int\vartheta_{n}^{2}\lvert\nabla_{x}^{2}u\rvert^{2}\,\dif\pi+\frac{C_{1}^{2}n^{-2}}{4\varepsilon}\int\lvert\nabla_{x}u\rvert^{2}\,\dif\pi\label{eq:boundforI2}\\
\leq & \varepsilon\|\vartheta_{n}\nabla_{x}^{2}u\|_{2}^{2}+\frac{C_{1}^{2}n^{-2}}{4\varepsilon}\lVert f\rVert_{2}^{2}.
\end{align}
Combining (\ref{eq:I1I2I3}), (\ref{eq:boundforI1}), (\ref{eq:boundforI3}),
(\ref{eq:boundforI2}) we have
\[
\|\vartheta_{n}\nabla_{x}u\|_{\pi}^{2}+(1-\varepsilon)\|\vartheta_{n}\nabla_{x}^{2}u\|_{2}^{2}\leq\left(2+c_{U}+\frac{C_{1}^{2}n^{-2}}{4\varepsilon}\right)\lVert f\rVert_{2}^{2}.
\]
Finally taking $\varepsilon=1/4$ we obtain
\[
\|\vartheta_{n}\nabla_{x}^{2}u\|_{2}^{2}\leq2(2+c_{U}+C_{1}^{2}n^{-2})\lVert f\rVert_{2}^{2}.
\]
The result follows by choosing $\vartheta$ as in Remark \ref{rem:cut-off-decreasing-away-origin},
letting $n\rightarrow\infty$ and invoking the monotone convergence
theorem.

\section*{Acknowledgements}

We would like to thank the Heilbronn Institute for Mathematical Research
Research for funding the Hypocoercivity Workshop held at the University
of Bristol, March 2020, which initiated this research. We would like
to thank Joris Bierkens, Anthony Lee and Sam Power for interesting
discussions related to this work. Research of CA supported by EPSRC
grants Bayes4Health, `New Approaches to Bayesian Data Science: Tackling
Challenges from the Health Sciences' (EP/R018561/1) and `CoSInES (COmputational
Statistical INference for Engineering and Security)' (EP/R034710/1).
Research of PD supported by the research programme `Zigzagging through
computational barriers' with project number 016.Vidi.189.043, financed
by the Dutch Research Council (NWO). Research of AQW supported by
EPSRC grant CoSInES (EP/R034710/1).

\bibliographystyle{plain}
\bibliography{bib-poly}

\end{document}